\documentclass[a4paper]{article}

\usepackage[utf8]{inputenx}

\usepackage{amsmath, amsthm}
\usepackage{mathtools}
\usepackage{amssymb}  
\usepackage{mathrsfs} 
\usepackage{bm}       
\usepackage{calc}     
\usepackage{hyperref}
\usepackage[final]{fixme}
\fxusetheme{color}

\usepackage[only,llbracket,rrbracket]{stmaryrd}
\SetSymbolFont{stmry}{bold}{U}{stmry}{m}{n}

\usepackage{silence}
\usepackage[eprint,doi=false,url=false,safeinputenc,maxnames=10]{biblatex}
\usepackage{style}

\title{Quantitative non-divergence and lower bounds
       for points with algebraic coordinates near manifolds}
\author{Alessandro Pezzoni
\footnote{Supported by EPSRC International Doctoral Scholars IDS Grant
EP/N509802/1.}}
\date{}

\begin{document}
\maketitle

\begin{abstract}
  Point counting estimates are a key stepping stone to various results in metric
  Diophantine approximation. In this paper we use the quantitative
  non-divergence estimates originally developed by Kleinbock and Margulis to
  improve lower bounds by Bernik, G\"{o}tze et al. for the number of points with
  algebraic conjugate coordinates close to a given manifold. In the process, we
  also improve on a Khinchin-Groshev-type theorem for a problem of constrained
  approximation by polynomials.
\end{abstract}

\listoffixmes

\section{Introduction}%
\label{sec:introduction}

In the course of developing his classification of real numbers, Mahler
conjectured that for every $\varepsilon > 0$ and Lebesgue almost every $x \in
\R$ the inequality
\begin{equation}
  \label{eq:mahler_conjecture}
  \abs{P(x)} < \polyheight{P}^{-n-\varepsilon}
\end{equation}
has at most finitely many solutions $P \in \Z[X]$ with
$\deg(P) \leq n$, where $\polyheight{P}$ denotes the \emph{(naive) height} of
$P$, i.e. the maximum of its coefficients in absolute value. This was later
proved by Sprind\v{z}uk \cite{Sprin1969}, and it marked the beginning of the
theory of Diophantine Approximation of dependent quatities, i.e. the study of
the Diophantine properties of points bound to a given manifold.

It is then natural to wonder about the Diophantine properties of the solutions
to a system of symultaneous equations of type \eqref{eq:mahler_conjecture} in
multiple independent variables $x_0,\dotsc,x_m \in \R$, i.e.
\begin{equation}
  \label{eq:mahler_higher_dimension}
  \max_{0 \leq k \leq m} \abs{P(x_k)} < \approxf\paren{\polyheight{P}}
\end{equation}
for some function $\approxf \colon \R^+ \to \R^+$, with solutions in integer
polynomials $P$ of degree between $m+1$ and $n$.
Indeed, in \cite[Problem~C]{Sprin1965} Sprind\v{z}uk conjectured that the
maximum $v > 0$ for which \eqref{eq:mahler_higher_dimension} with
$\approxf(Q) = Q^{-v}$ has infinitely many solutions for all $x$ in a set of
positive measure is
\begin{equation}
  \label{eq:exponent_sprindzuk_conjecture}
  v = \frac{n+1}{m+1} - 1,
\end{equation}
and this was later proved by Bernik in \cite{Berni1977}.

The problem \eqref{eq:mahler_higher_dimension} was then considered for abitrary
$\approxf$ and
$m = 1$ in \cite{BerniB1997}, as well as for the case where the variables $x_k$
can also take complex or $p$-adic values in \cite{BerniBD2008,BerniBD2010}
($m = 2$) and \cite{Budar2011a,Budar2011b} (arbitrary $m$). In particular, the
following result is contained in the preprint \cite{BeresBB2017}, which deals
with the more general case of systems of linear forms in dependent variables,
i.e.
\begin{equation}
  \label{eq:small_linear_forms}
  \max_{0 \leq k \leq m}
    \abs{\bm{a} \cdot \bm{f}_k(\bm{x}_k)}
    < \approxf\paren{\norm{\bm{a}}}
\end{equation}
with solutions in $\bm{a} \in \Z^{n+1}$, where $\bm{f_k} \colon U_k \to \R^n$
are sufficiently regular maps defined on open balls $U_k \subset \R^{d_k}$ and
$\approxf \colon \R^+ \to \R^+$ as before. Here and throughout this paper
$\norm{\cdot}$ will denote the sup norm $\R^n$ unless otherwise specified,
although note that most of the results presented here still hold with minor
modification for any other choice of norm.

\begin{theorem}[{\cite[Theorem~1]{BeresBB2017}}]
  \label{thm:independent_polynomial_khinchin}
  Consider integers $n > m \geq 0$, a function $\approxf \colon \R^+ \to \R^+$,
  and a ball $B \subset \R^{m+1}$. Let $\Leb{}$ be the Lebesgue measure on
  $\R^{m+1}$. Then
  \[
    \Leb{\dualappr_{n,m+1}(\approxf) \cap B} =
    \begin{cases}
      0 & \text{if } S_{n,m+1}(\approxf) < \infty \\
      \Leb{B} & \text{if } \approxf \text{ is monotonic and }
        S_{n,m+1}(\approxf) = \infty
    \end{cases}
  \]
  where $\dualappr_{n,m+1}(\approxf)$ denotes the set of
  $(x_0,\dotsc,x_m) \in \R^{m+1}$ which satisfy
  \eqref{eq:mahler_higher_dimension} for infinitely many polynomials $P$ of
  degree up to $n$, and where
  \[
    S_{n,m+1}(\approxf)
    \coloneqq \sum_{Q = 1}^\infty Q^{n-m-1} \approxf^{m+1}(Q).
  \]
\end{theorem}

\begin{note}
  Like many other Khinchin-Groshev-type theorems, this kind of result has already found
  applications in communication engineering, specifically in the field of
  \emph{interference alignment}; see for example \cite[Appendix~B]{OrdenE2013}
  and \cite[Section~IV]{MahMK2010}, or \cite{GhaseMK2010,MotahOMK2014} for
  examples which require results of approximation on manifolds. The interested
  reader may also find a more accessible description of how Khinchin-like
  theorems come into play in the theory of interference alignment in
  \cite[Appendix~A]{AdiceBLVZ2016}.
\end{note}

Finally, one might consider what changes after introducing a dependency among
the variables $x_0,\dotsc,x_m$ of \eqref{eq:mahler_higher_dimension} (or
$\bm{x}_0,\dotsc,\bm{x}_m$ in \eqref{eq:small_linear_forms}), i.e. when they are
parametrised by a sufficiently regular map
$\bm{f} \colon \domf \subset \R^d \to \R^{m+1}$, and this is the subject of the
present paper.

Clearly, if $P(x)$ is small, then $x$ must be close to at least one of the roots
of $P$. In particular, this means that if $P$ is irreducible and
$\abs{P(f_0(\bm{x}))},\dotsc,\abs{P(f_m(\bm{x}))}$ are all small, then there
must be a point $\bm{\alpha} \in \R^{m+1}$ close to $\bm{f}(\bm{x})$, where the
coordinates of $\bm{\alpha}$ are algebraic and conjugate. Note, however, that
there are subtle differences among these two types of approximation, as evinced
by the difference between the classifications of numbers of Mahler and Koksma
\cite[Section~3.4]{Bugea2004}.

Nonetheless, a good first step towards establishing a result like
Theorem~\ref{thm:independent_polynomial_khinchin} is to provide an estimate for
the number of such points $\bm{\alpha}$ which are sufficiently close to the
manifold $\manifold$ parametrised by $\bm{f}$ (see e.g.
\cite[Section~2.6]{Sprin1979}). Furthermore, the techniques used to derive such
estimates can be of interest in and of themselves; for example, in the case of
rational points they have been adapted to derive an efficient algorithm to
compute the rational points with bounded denominator on a given manifold, see
\cite{Elkie2000}, or \cite[Section~11]{Mazur2004} for a nice overview.
This problem was first considered for planar curves by Bernik, G\"{o}tze and
Kukso in \cite{BerniGK2014}. In other words, let $\domf \subset \R$ be a bounded
open interval and let $f_1 \colon \domf \to \R$ be a $\cont[1]$ function; also,
define the sets
\begin{equation}
  \label{eq:original_set_of_algebraic_points}
  \begin{gathered}
    \alg{2}{n}{Q}
    \coloneqq \{\bm{\alpha} \in \R^2 :
                \bm{\alpha} \text{ is algebraic },
                \deg(\bm{\alpha}) \leq n,
                \polyheight{\bm{\alpha}} \leq Q
              \} \\
    M_{f_1}^n(Q,\gamma,\domf)
    \coloneqq \{(\alpha_0,\alpha_1) \in \alg{2}{n}{Q} :
                \alpha_0 \in \domf,
                    \abs{f_1(\alpha_0) - \alpha_1} < \capprcount Q^{-\gamma}
              \},
  \end{gathered}
\end{equation}
where $\capprcount > 0$ is fixed. Here by $\bm{\alpha} \in \R^{m+1}$
\emph{algebraic} we mean that its coordinates are algebraic conjugate real
numbers, and by $\polyheight{\bm{\alpha}}$ we denote the height of their
minimal polynomial.

A lower bound for $\# M_{f_1}^n(Q,\gamma,\domf)$ was provided in
\cite{BerniGK2014} for $0 < \gamma < \frac{1}{2}$. This was soon extended in
\cite{BerniGG2016}, where Bernik, G\"{o}tze and Gusakova also provided an upper
bound. We also note that recently Bernik, Budarina and Dickinson provided
an analogous lower bound for surfaces in $\R^3$ \cite{BerniBD2020}.

\begin{theorem}[{\cite[Theorem~1]{BerniGG2016}}]
  \label{thm:bernik_goetze_gusakova}
  Suppose that both $\sup_\domf \abs{f_1'}$ and $\#\{x \in \domf : f_1(x) = x\}$
  are bounded. If $\capprcount$ is sufficiently large, then
  \[
    \# M_{f_1}^n(Q, \gamma, \domf) \asymp Q^{n+1-\gamma}
  \]
  for every $Q$ large enough and $0 < \gamma < 1$.
\end{theorem}

\begin{note}
  Here and throughout the paper we will make extended use of Vinogradov's
  notation. Namely, we will write $a \ll b$ if there is a constant $c > 0$ such
  that $a < c b$, as well as $b \gg a$ if $a \ll b$, and $a \asymp b$ when
  $a \ll b$ and $b \ll a$ symultaneously, in which case we say that $a$ is
  \emph{comparable} to $b$. Occasionally we will make dependencies
  of the implied constant explicit via a subscript, e.g. $a \ll_\varepsilon b$,
  and $c$ is generally assumed to be independent of the variables that $a$ and
  $b$ depend on, although it could depend on the other parameters involved.
  We also extend this notation to vectors in the natural way: if
  $\bm{a} = (a_1,\dotsc,a_r)$ and $\bm{b} = (b_1,\dotsc,b_r)$, then
  $\bm{a} \ll \bm{b}$ means that $a_i \ll b_i$ for every $1 \leq i \leq r$, and
  similarly for $\gg$ and $\asymp$.
\end{note}

In the present paper we will extend the lower bound in
Theorem~\ref{thm:bernik_goetze_gusakova} to sufficiently regular manifolds in
arbitrary dimension.
While the characterisation of these manifolds is quite technical, as a special
case our results hold true when $\bm{f}$ is analytic with algebraically
independent components.
In particular, the following is a special case of
Theorem~\ref{thm:lower_bound_for_manifold_count}, which extends the range of
$\gamma$ to the best possible --- here $M_{\bm{f}}^n$ is the higher dimensional
analogue of $M_{f_1}^n$, see \eqref{eq:general_set_of_algebraic_points} for
details.

\begin{theorem}
  \label{thm:simplified_lower_bound}
  Let $\domf \subset \R^d$ be a bounded open set, and let
  $\bm{f}(\bm{x}) = (x_0,\dotsc,x_{d-1},f_d(\bm{x}),\dotsc,f_m(\bm{x}))$ be an
  analytic function $\domf \to \R^{m+1}$ with algebraically independent
  components. Then for $\capprcount > 0$ fixed and for every
  \begin{equation*}
    0 < \gamma \leq \frac{n + 1}{m + 1}
  \end{equation*}
  we have
  \[
    \#M_{\bm{f}}^n(Q, \gamma, \domf)
    \gg Q^{n + 1 - \gamma (m + 1 - d)}
  \]
  for every $Q$ sufficiently large, where the implied constant does not depend
  on $Q$.
\end{theorem}

In the process of proving Theorem~\ref{thm:simplified_lower_bound}, we will also
be able to extend the divergence part of
Theorem~\ref{thm:independent_polynomial_khinchin} as follows; here
$\HausdorffM^s$ denotes the usual Hausdorff $s$-measure (see
Definition~\ref{ex:hausdorff_measure}).

\begin{theorem}
  \label{thm:simplified_polynomial_jarnik_divergence}
  Let $\domf$, $\bm{f}$ be as in Theorem~\ref{thm:simplified_lower_bound}, and
  let $\approxf \colon \R^+ \to \R^+$ be a decreasing function such that
  $\approxf(Q) \gg Q^{\frac{m-n}{m+1}}$.
  Further, denote by $\dualappr_{n,\bm{f}}(\approxf)$ the set of
  $\bm{x} \in \domf$ such that $\bm{f}(\bm{x}) \in \dualappr_{n,m+1}(\approxf)$.
  Then for any $0 < s \leq d$ we have
  \[
    \Haus[s]{\dualappr_{n,\bm{f}}(\approxf)} = \Haus[s]{\domf}
    \quad \text{if} \quad
    \sum_{Q=1}^\infty
      Q^{n-m-1}
      \approxf(Q)^{m+1}
      \paren{\frac{\approxf(Q)}{Q}}^{s-d}
    = \infty.
  \]
\end{theorem}

\begin{corollary}[Cfr. {\cite[Corollary~1]{BeresBB2017}}]
  \label{cor:polynomial_hausdorff_dimension}
  Let $\Hausdim$ denote Hausdorff dimension. In the same setting of
  Theorem~\ref{thm:simplified_polynomial_jarnik_divergence}, we have that
  \[
    \Hausdim\paren{\dualappr_{n,\bm{f}}(\approxf)}
    \geq \min\left\{ d, \frac{n + 1}{\tau_{\approxf} + 1} + d - m - 1 \right\},
  \]
  where
  \[
    \tau_{\approxf}
    \coloneqq \liminf_{Q \to \infty} - \frac{\log \approxf(Q)}{\log Q}
  \]
  is the lower order of $\approxf^{-1}$ at infinity.
\end{corollary}

\begin{remark}
  The condition $\approxf(Q) \gg Q^{\frac{m-n}{m+1}}$ of
  Theorem~\ref{thm:simplified_polynomial_jarnik_divergence} implies
  that $\tau_{\approxf} \leq \frac{n-m}{m+1}$, which is precisely the situation
  where
  \[
    \frac{n + 1}{\tau_{\approxf} + 1} + d - m - 1 \geq d.
  \]
  Therefore in this setting we actually have
  $\Hausdim\paren{\dualappr_{n,\bm{f}}(\approxf)} = d$. On the other hand,
  Theorem~\ref{thm:polynomial_jarnik_divergence} shows that the picture is more
  interesting with two separate approximation functions on $\R^d$ and
  $\R^{m+1-d}$.
\end{remark}

Our proof exploits the powerful quantitative non-divergence bounds first
introduced by Kleinbock and Margulis in \cite{KleinM1998}, and it has a
similar flavour to \cite{BeresBB2017}. This paper is structured as follows:
\begin{itemize}
  \item in the next section we will describe our setting and main results;
  \item in sections \ref{sec:schur_polynomials} and \ref{sec:good_functions} we
    will discuss the regularity conditions that these depend on and provide some
    examples of functions that satisfy them;
  \item the next sections are devoted to the proofs, in this order: sections
    \ref{sec:conjugate_coordinates} and \ref{sec:tailored_polynomials} for the
    extension of Theorem~\ref{thm:bernik_goetze_gusakova},
    section~\ref{sec:ubiquity} for the extension of
    Theorem~\ref{thm:independent_polynomial_khinchin}, and
    section~\ref{sec:proof_of_the_main_theorem} for the proof of
    Theorem~\ref{thm:bound_over_an_interval}, which underpins the whole
    argument; and
  \item the last section contains some final remarks about possible directions
    in which this work could be extended.
\end{itemize}

The author would like to thank Prof. Victor Beresnevich for the continued
support and insightful comments, and without whom this work would not have been
possible.

\section{The main result}
\label{sec:main_result}

Let $X$ be a metric space. If $\kappa > 0$ and $B \subset X$ is a ball centred
at $x$ and with radius $r$, throughout this paper $\kappa B$ will denote the
dilation of $B$ by $\kappa$, i.e. the ball with centre $x$ and radius
$\kappa r$.

\begin{definition}
  \label{def:besicovitch_metric_space}
  Let $N > 0$. Following \cite{KleinLW2004}, a metric space $X$ is called
  \emph{$N$-Besicovitch} if, for any bounded set $A \subset X$ and any
  collection of balls $\mathcal{B}$ such that every $x \in A$ is in the centre
  of a ball in $\mathcal{B}$, there is a countable collection
  $\Omega \subseteq \mathcal{B}$ which covers $A$ and such that every point
  $x \in A$ lies in at most finitely many balls in $\Omega$. We will also say
  that $X$ is \emph{Besicovitch} if it is $N$-Besicovitch for some $N > 0$.
\end{definition}

\begin{example}
  \label{ex:euclidean_space_is_besicovitch}
  It is well known that $\R^n$ with the Euclidean metric is Besicovitch, see
  e.g. \cite[Theorem~2.7]{Matti1995}.
\end{example}

\begin{definition}
  \label{def:measure_properties}
  Let $U \subset X$ be an open subset and let $\nu$ be a Radon measure on $U$.
  Following \cite{KleinLW2004}, we will say that $\nu$ is:
  \begin{itemize}
    \item \emph{$D$-Federer (or doubling) on $U$} for some $D > 0$ if
      \[
        \nu\paren{3^{-1} B} > \frac{\nu(B)}{D}
      \]
      for any ball $B \subset U$ centred on $\supp \nu$.
    \item \emph{Federer} if for $\nu$ almost every $x \in X$ there are a
      neighbourhood $U$ of $x$ and a $D > 0$ such that $\nu$ is $D$-Federer on
      $U$.
    \item \emph{(Rationally) non-planar} if $X = \R^d$ and $\nu(\plane) = 0$ for
      every (rational) affine hyperplane of $\R^d$.
    \item \emph{$(C,\alpha)$-decaying on $U$} for some $C,\alpha > 0$
      if $X = \R^d$ and for any ball $B$ centred on $\supp \nu$, any affine
      hyperplane $\plane \subset \R^d$, and any $\varepsilon > 0$ we have
      \begin{equation}
        \label{eq:decaying_measure}
        \nu\paren*[\big]{B \cap \plane^{(\varepsilon)}}
        \leq C \paren{\frac{\varepsilon}{\norm{d_\plane}_{\nu,B}}}^\alpha \nu(B),
      \end{equation}
      where $\plane^{(\varepsilon)}$ is the $\varepsilon$-neighbourhood of
      $\plane$, $d_\plane$ is the Euclidean distance from $\plane$, and
      $\norm{d_\plane}_{\nu,B} = \sup_{\bm{x} \in B \cap \supp\nu}
      d_\plane(\bm{x})$. C.f. Definition~\ref{def:good_function}.
    \item \emph{Absolutely $(C,\alpha)$-decaying on $U$} if
      \eqref{eq:decaying_measure} holds with the radius of $B$ in place of
      $\norm{d_\plane}_{\nu,B}$.
    \item \emph{(Absolutely) decaying} if for $\nu$ almost every $x \in X$ there
      are a neighbourhood $U$ of $x$ and constants $C,\alpha > 0$ such that
      $\nu$ is (absolutely) $(C,\alpha)$-decaying on $U$.
  \end{itemize}
\end{definition}

\begin{remark}
  Both classes of Federer and absolutely decaying measures are closed under
  restriction to open subsets $U \subset X$. Furthermore, they are also closed
  with respect to taking finite products \cite[Theorem~2.4]{KleinLW2004}.
\end{remark}

\begin{example}
  \label{ex:federer_and_decaying_measures}
  Examples of measures that are Federerer and absolutely decaying on $\R^d$
  include the Lebesgue measure and measures supported on certain self-similar
  sets (see e.g. \cite{KleinLW2004} and \cite{MauldU2003}).
\end{example}

Now consider a $d$-dimensional manifold $\manifold$ in $\R^{m+1}$,
parametrised over a bounded open subset $\domf \subset X$ by a continuous map
$\bm{f}(\bm{x}) = (f_0(\bm{x}), \dotsc, f_m(\bm{x}))$. Without loss of
generality, if $X = \R^d$ we will assume that $f_i(\bm{x}) = x_i$ for
$0 \leq i < d$ and write $\partf$ for $(f_d,\dotsc,f_m)$.
Then, for $n \geq m+1$ fixed, define the vectors in $\R^{n+1}$
\[
  \begin{array}{c c}
    \begin{aligned}
      \bm{v}_i = \bm{v}_i(\bm{x}) &\coloneqq
      \begin{pmatrix}
        1 &
        f_i(\bm{x}) &
        f_i(\bm{x})^2 &
        \cdots &
        \mathmakebox[\widthof{$n f_i(\bm{x})^{n-1}$}][l]{f_i(\bm{x})^n}
      \end{pmatrix} \\
      \bm{v}'_i = \bm{v}'_i(\bm{x}) &\coloneqq
      \begin{pmatrix}
        0 &
        \mathmakebox[\widthof{$f_i(\bm{x})$}]{1} &
        \mathmakebox[\widthof{$f_i(\bm{x}^2)$}][l]{2f_i(\bm{x})} &
        \cdots &
        n f_i(\bm{x})^{n-1}
      \end{pmatrix}
    \end{aligned}
    & \quad 0 \leq i \leq m
  \end{array}
\]
and the $(n+1) \times (n+1)$ matrices
\begin{align}
  \label{eq:main_matrix}
  M_{\bm{f}} &\coloneqq
  \left(
  \begin{array}{@{}c@{}}
    \begin{matrix}
      \bm{v}_0 \\
      \vdots \\
      \bm{v}_m
    \end{matrix} \\
    \begin{array}{c | c}
      \hline \\[\dimexpr-\normalbaselineskip+2pt]
      \bigzero & \bigid_{n-m}
    \end{array}
    \\[.5ex]
  \end{array}
  \right) \\
  \label{eq:system_matrices}
  U_{\bm{f}}^h &\coloneqq
  \left(
  \begin{array}{@{}c@{}}
    \begin{matrix}
      \bm{v}_0 \\
      \vdots \\
      \bm{v}_m \\
      \bm{v}'_h
    \end{matrix} \\
    \begin{array}{c | c}
      \hline \\[\dimexpr-\normalbaselineskip+2pt]
      \bigzero & \bigid_{n-m-1}
    \end{array}
    \\[.5ex]
  \end{array}
  \right).
\end{align}

\begin{remark}
  \label{rmk:non_zero_determinant}
  The determinant of $U_{\bm{f}}^h$ is the same as the determinant of the
  submatrix $\tilde{U}_{\bm{f}}^h$ formed by its first $n+1$ rows and columns.
  The latter is an example of what in the literature is known as a
  \emph{confluent Vandermonde matrix}, and a theorem of Schendel's (see e.g.
  \cite[Theorem~20]{Kratt1999}) shows that
  \[
    \abs*[\big]{\det \tilde{U}_{\bm{f}}^h}
    = \prod_{0 \leq i < j \leq m} \abs{f_i(\bm{x}) - f_j(\bm{x})}^{e_i e_j},
  \]
  where $e_i$ is $2$ if $i = h$ and $1$ otherwise. In particular,
  $\det U_{\bm{f}}^h \neq 0$ if and only if the Vandermonde
  polynomial $\Vand(\bm{f})$ is non-zero (see \eqref{eq:vandermonde_polynomial}
  below).
\end{remark}

For ease of notation, we will write $\intrange{n}$ instead of $\{0,\dotsc,n\}$,
as well as $\indexsets{n}{\tau}$ for the set of
$I = (i_1,\dotsc,i_\tau) \in \intrange{n}^\tau$ such that
$i_1 < i_2 < \dotsb < i_\tau$, where $1 \leq \tau \leq n+1$.
Given an $(n+1) \times (n+1)$ matrix $A$, we will also write $A_{I,J}$ for the
submatrix of $A$ with rows indexed by $I \subseteq \indexsets{n}{\tau}$ and columns
indexed by $J \subseteq \indexsets{n}{\tau}$, and $\minor{A}{I,J}$ for its
determinant.

Then, for every $1 \leq \tau \leq n+1$ and for every $I \in \indexsets{n}{\tau}$,
define the map from the set $\mat{n+1}{n+1}$ of $(n+1) \times (n+1)$ matrices to
$\bigwedge^\tau \R^{n+1} \simeq \R^{\binom{n+1}{\tau}}$ given by
\begin{equation}
  \label{eq:grassmanian_map}
  \grass[I] \colon
  A \mapsto \paren*[\big]{\minor{A}{I,J}}_{J \in \indexsets{n}{\tau}}.
\end{equation}
In other words, $\grass[I][A]$ is the image under the Pl\"ucker embedding of the
linear subspace of $\R^{n+1}$ spanned by the rows of $A$ indexed by $I$.
Furthermore, in Section~\ref{sec:schur_polynomials} we will see that for
$I \in \indexsets{m}{\tau}$ and $1 \leq \tau \leq m+1$ we have
\[
  \grass[I][M_{\bm{f}}] =
  \paren*[\big]{\Vand(\bm{f}_I) s_{\lambda}(\bm{f})}_{
               \substack{\abs{\lambda} \leq n+1-\tau \\
                         \len{\lambda} \leq \tau}}
\]
where $\Vand(\bm{f}_I)$ is the Vandermonde polynomial of
$\bm{f}_I = (f_i)_{i \in I}$, i.e.
\begin{equation}
  \label{eq:vandermonde_polynomial}
  \Vand(\bm{f}_I) \coloneqq \prod_{\substack{i,j \in I \\ i < j}} (f_j - f_i),
\end{equation}
and $s_{\lambda}$ is the \emph{Schur polynomial}
in $\tau$ indeterminates correponding to the partition $\lambda$ of the integer
$\abs{\lambda}$ with $\len{\lambda}$ parts (see
Definition~\ref{def:schur_polynomial}). Therefore, we also define
\begin{equation}
  \label{eq:schur_map}
  \Schur[n,\tau] \colon \bm{T} = (T_1,\dotsc,T_{\tau})
                 \mapsto \paren*[\big]{s_{\lambda}(\bm{T})}_{
                                      \substack{\abs{\lambda} \leq n+1-\tau \\
                                      \len{\lambda} \leq \tau}}
\end{equation}
and with a slight abuse of notation we will write $\Schur[\tau][\bm{T}]$ or
$\Schur[][\bm{T}]$ instead of $\Schur[n,\tau][\bm{T}]$ whenever $n$ or $\tau$ are
clear from the context. Finally, observe that if $\Vand(\bm{f})$ is bounded on
$\domf$, then $\grass[I][M_{\bm{f}}] \asymp \Schur[][\bm{f}_I]$.

\begin{definition}
  \label{def:non_symmetric}
  Let $X$ be a measure space and $\nu$ a measure on $X$. Fix $\tau \geq 1$ and
  let $\sympolybounded{\tau}{k}$ be the space of rational symmetric polynomials
  of degree up to $k$. Given a map
  $\bm{f} \colon \domf \subseteq X \to \R^\tau$, the pair $(\bm{f}, \nu)$ is
  called:
  \begin{itemize}
    \item \emph{Non-symmetric (of degree $k$) at $x$} if for every neighbourhood
      $B \ni x$ and $s \in \sympolybounded{\tau}{k}$ we have that
      $\bm{f}(B \cap \supp\nu)$ is not contained in the zero locus of
      $s$ in $\R^{\tau}$.
      Cf. the definitions of non-planarity from \cite{Klein2008,KleinT2007}.
    \item \emph{Non-symmetric (of degree $k$) on $\domf$} if it is non-symmetric
      of degree $k$ at every $x \in \domf \cap \supp\nu$.
    \item \emph{Symmetrically good (of degree $k$) on $\domf$} if it is
      non-symmetric (of degree $k$) on $\domf$ and there are constants
      $C,\alpha > 0$ such that $s(\bm{f})$ is $(C,\alpha)$-good on $B$ with
      respect to $\nu$ for every $s \in \sympolybounded{\tau}{k}$ (see
      Definition~\ref{def:good_function}).
  \end{itemize}
\end{definition}

\begin{note}
   See Corollary~\ref{cor:symmetric_goodness_equivalence} for an equivalent
   characterisation of $(\bm{f}, \nu)$ being non-symmetric of degree $k$ in
   terms of the components of an appropriate $\Schur(\bm{f})$.
\end{note}

Now take functions
$\psi_0, \dotsc, \psi_{m},\varphi_{m+1}, \dotsc, \varphi_{n} \colon \R^+ \to \R^+$,
and for $Q > 0$ consider the system of inequalities
\begin{equation}
  \label{eq:d_dimensional_inequalities}
  \begin{cases}
    \abs{P(f_k(\bm{x}))} < \psi_k(Q) & \text{for } 0 \leq k \leq m \\
    \max_{i}\abs{P'(f_i(\bm{x}))} \leq \varphi_{m+1}(Q) \\
    \abs{a_k} \leq \varphi_k(Q) & \text{for } m+1 < k \leq n
  \end{cases}
\end{equation}
with solutions in integer polynomials $P = a_n X^n + \dotsc + a_0$
of degree at most $n$. Our main result concerns the set
\begin{align*}
  \maxderset^n_{\bm{f}}(Q, \domf)
  = \maxderset^n_{\bm{f}}(Q, \domf; \psi_0, \dotsc, \psi_m, \varphi_{m+1}, \dotsc, \varphi_n)
\end{align*}
of points $x \in \domf$ for which \eqref{eq:d_dimensional_inequalities} admits a
solution.
For ease of notation, given $I \in \indexsets[0]{m}{\tau_1}$ and
$J \in \indexsets[m+1]{n}{\tau_2}$, let
\[
  \bm{\psi}_I \defeq \prod_{i \in I} \psi_i, \quad
  \bm{\varphi}_J \defeq \prod_{j \in J} \varphi_j,
\]
as well as $\bm{\psi} \defeq \bm{\psi}_{\intrange[0]{m}}$ and
$\bm{\varphi} \defeq \bm{\varphi}_{\intrange[m+1]{n}}$.

\begin{theorem}
  \label{thm:bound_over_an_interval}
  Let $X$ be an $N$-Besicovitch space, $\domf \subset X$ a bounded open subset,
  and let $\nu$ be a $D$-Federer measure on $\domf$. Let $\bm{f} \colon \domf
  \to \R^{m+1}$ be a continuous function such that
  $\vandmax \geq \Vand(\bm{f}) \geq \vandmin$ on $\domf \cap \supp \nu$, where
  $\Vand(\bm{f})$ is the Vandermonde polynomial of $\bm{f}$.
  Furthermore, let $\bm{\psi},\bm{\varphi}$ be as above, and suppose that for
  some $n > 0$
  \begin{equation}
    \label{eq:cnd:derivative_bound}
    \varphi_{m+1}(Q)^{n+1} \gg \bm{\psi}(Q)\bm{\varphi}(Q)
  \end{equation}
  and that for every $1 \leq \tau \leq m+1$ there is a choice of
  $I \in \indexsets{m}{\tau}$ such that
  \begin{gather}
    \label{eq:cnd:t_positive_d}
    \bm{\psi}(Q)\bm{\varphi}(Q) \gg \psi_I(Q)^{\frac{n+1}{\tau}}, \text{ and} \\
    \label{eq:cnd:minors_d}
    (\bm{f}_I, \nu) \text{ is symmetrically }
    (\tilde{C},\alpha)\text{-good of degree } n+1-\tau \text{ on } \domf
  \end{gather}
  for some $\tilde{C},\alpha > 0$. Then for any $0 < \theta < 1$ and for $Q$
  large enough we may find a subset $\domfgood{\theta} \subset \domf$ with
  measure $\nu(\domfgood{\theta}) > \theta \nu(\domf)$, as well as
  $C = C(\tilde{C}, \vandmax, \vandmin, n, N, D) > 0$ and
  $\rho = \rho(\bm{f}, n, \domfgood{\theta}) > 0$, such that
  \[
    \nu\paren{\maxderset^n_{\bm{f}}(Q, \domfgood{\theta})}
    \leq C \paren{ \frac{\bm{\psi}(Q) \bm{\varphi}(Q)}
                        {\rho^{n+1}} }^\frac{\alpha}{n+1}
         \nu(\domfgood{\theta}).
  \]
\end{theorem}

\begin{remark}
  \label{rmk:compactness_of_domfgood}
  The set $\domfgood{\theta}$ can be chosen to be either compact or a union of
  finitely many open balls (which form a cover for this compact set).
\end{remark}

\begin{note}
  Corollaries \ref{cor:analytic_symmetric_independence} and
  \ref{cor:analytic_nondegenerate} below show that it is relatively
  straighforward to check condition~\eqref{eq:cnd:minors_d} when $X =
  \R^d$ and $\bm{f}$ is analytic.
\end{note}

\begin{corollary}
  \label{cor:metric_convergence}
  Let
  $\maxderset^n_{\bm{f}}(\domf)
  = \limsup_{Q} \maxderset^n_{\bm{f}}(Q, \domf)$.
  Under the hypothesis of Theorem~\ref{thm:bound_over_an_interval},
  further assume that $X = \R^d$, $\bm{f}$ is analytic and $\nu$ is the Lebesgue
  measure on $\domf$. Then $\alpha$ can be chosen to be one of:
  \begin{itemize}
    \item $1/d(N-1)$, or
    \item $1/k\deg(\bm{f})$ if $\bm{f}$ is a polynomial map.
  \end{itemize}
  Moreover,
  \begin{equation}
    \label{eq:cnd:metric_convergence}
    \sum_{Q = 1}^{\infty}
    \paren*[\big]{\bm{\psi}(Q) \bm{\varphi}(Q)}^{\frac{\alpha}{n+1}} < \infty
    \qquad \text{implies} \qquad
    \nu\paren*[\big]{\maxderset^n_{\bm{f}}(\domf)} = 0.
  \end{equation}
\end{corollary}


\begin{corollary}
  \label{cor:n+1_points}
  Consider $\psi_0,\dotsc,\psi_{m}$,
  $\bm{f}$ as in Theorem~\ref{thm:bound_over_an_interval} such that
  condition~\eqref{eq:cnd:t_positive_d} holds with
  $\varphi_{m+1}(Q) = \dotsb = \varphi_n(Q) = Q$.
  Furthermore, assume that for every $0 \leq i \leq m$ and for $Q$ large enough
  $\psi_i(Q)Q^{-1}$ is decreasing, and that there are constants
  $\deltamin, \deltamax > 0$ such that
  \begin{equation*}
    \deltamin \leq \bm{\psi}(Q) Q^{n-m} \leq \deltamax.
  \end{equation*}
  Then for every $0 < \theta < 1$ there are a constant $c > 0$ and a subset
  $\domfgood{\theta}$ of $\domf$, independent of $Q$, such that
  $\nu(\domfgood{\theta}) > \theta \nu(\domf)$ and every
  $\bm{x} \in \domfgood{\theta}$ admits $n+1$ distinct points
  $(\alpha_0, \dotsc, \alpha_m) \in \R^{m+1}$ with algebraic conjugate
  coordinates of height $\polyheight{\alpha_k} \ll Q$ which satisfy
  \begin{equation}
    \label{eq:closeness_tolerance}
    \abs{f_k(\bm{x}) - \alpha_k} < c \frac{\psi_k(Q)}{Q}
    \qquad \text{for } 0 \leq k \leq m
  \end{equation}
  whenever $Q > 0$ is sufficiently large.
\end{corollary}

\begin{remark}
  \label{rmk:n+1_points_constant}
  Here $c$ can be chosen to be $\frac{\ctailoredy}{\ctailoredf}$, where the
  constants $\ctailoredy,\ctailoredf$ are the same as in
  Corollary~\ref{cor:tailored_polynomials}. In particular, it depends on
  $\deltamin$ and $\deltamax$ but not on the functions $\psi_i$ themselves.
\end{remark}

Now suppose that $\domf \subset \R^d$ and without loss of generality assume that
$f_i(\bm{x}) = x_i$ for each $0 \leq i < d$. Then for a given $\capprcount > 0$,
let (c.f.  \eqref{eq:original_set_of_algebraic_points})
\begin{equation}
  \label{eq:general_set_of_algebraic_points}
  \begin{gathered}
    \alg{m+1}{n}{Q}
    \coloneqq \{\bm{\alpha} \in \R^{m+1} :
                \bm{\alpha} \text{ is algebraic },
                \deg(\bm{\alpha}) \leq n,
                \polyheight{\bm{\alpha}} \leq Q
              \} \\
    \begin{aligned}
      M_{\bm{f}}^n(Q, \gamma, \domf) \coloneqq
      \big\{&{} (\alpha_0, \dotsc, \alpha_m) \in \mathbb{A}^{m+1}_n(Q) : \\
            &{} \bm{\alpha} = (\alpha_0, \dotsc, \alpha_{d-1}) \in \domf,
                \max_{d \leq j \leq m}
                  \abs{f_j(\bm{\alpha}) - \alpha_j} < \capprcount Q^{-\gamma}
      \big\}
    \end{aligned}
  \end{gathered}
\end{equation}
where and $\gamma > 0$. Then we are able to extend
the lower bound from Theorem~\ref{thm:bernik_goetze_gusakova} as follows:

\begin{theorem}
  \label{thm:lower_bound_for_manifold_count}
  Let $\bm{f} \colon \domf \to \R^{m+1}$ be a $\cont[1]$ map as above such that
  $\Vand(\bm{f}) \neq 0$, and assume that, up to reordering $f_d,\dotsc,f_m$,
  \begin{equation}
    \label{eq:cnd:count_bound}
    (\bm{x},f_d,\dotsc,f_\tau) \text{ is symmetrically good of degree }
                                      n-\tau \text{ on } \domf
  \end{equation}
  for every $d \leq \tau \leq m$. Then for $\capprcount > 0$ fixed and for every
  \begin{equation}
    \label{eq:lower_bound_epsilon}
    0 < \gamma \leq \frac{n + 1}{m + 1}
  \end{equation}
  we have
  \[
    \#M_{\bm{f}}^n(Q, \gamma, \domf)
    \gg Q^{n + 1 - \gamma (m + 1 - d)}
  \]
  for every $Q$ sufficiently large, where the implied constant does not depend
  on $Q$.
\end{theorem}

\begin{remark}
  Theorem~\ref{thm:lower_bound_for_manifold_count} is proved by taking
  $\psi_k = Q^{1-\gamma}$ for $d \leq k \leq m$ in
  Corollary~\ref{cor:n+1_points}. In particular, Sprind\v{z}uk's conjecture
  \eqref{eq:exponent_sprindzuk_conjecture} shows that the upper bound
  $\gamma \leq \frac{n+1}{m+1}$ is in general the best possible.
\end{remark}

\begin{remark}
  \label{rmk:upper_bound}
  If we could show that the lower bound for $Q$ and the implied constant in
  Theorem~\ref{thm:lower_bound_for_manifold_count} can be chosen independently
  of translations of $\manifold$, then we would also have an upper bound
  for $\# M_{\bm{f}}^n(Q, \gamma, \domf)$. Indeed, without loss of generality
  we may assume that $\bm{f}$ is bounded on $\domf$, hence $\manifold$ is
  contained in an open set $K$ of volume comparable to $\vol(\domf)$. Now let
  $\manifold_\gamma$ be the $\gamma$-neighbourhood of $\manifold$, i.e. the set
  \[
      \manifold_\gamma \coloneqq
      \big\{ (\bm{x},\bm{y}) \in \domf \times \R^{m+1-d} :
             \norm{\bm{f}(\bm{x}) - \bm{y}} < Q^{-\gamma}
      \big\}
  \]
  and note that
  $\vol(\manifold_\gamma) \asymp \vol(\domf) Q^{-\gamma(m+1-d)}$. In
  particular, up to replacing $K$ with a slightly bigger open set, we may
  assume that $K$ contains a union of disjoint translated copies
  $\{\manifold_\gamma^j\}_{j \in J}$ of $\manifold_\gamma$, with
  \[
    \#J \asymp \vol(K)/\vol(\manifold_\gamma) \asymp Q^{\gamma(m+1-d)}.
  \]
  If the implied constant in Theorem~\ref{thm:lower_bound_for_manifold_count}
  can be chosen to be in a translation invariant way, then we may find
  $c, Q_0 > 0$ such that for every $Q > Q_0$ and for every $j \in J$ we have
  \[
    \# \paren{\manifold_\gamma^j \cap \alg{m+1}{n}{Q}}
    > c Q^{n + 1 - \gamma (m + 1 - d)}.
  \]
  It follows that
  \[
    \# \paren{K \cap \alg{m+1}{n}{Q}}
    \gg Q^{n + 1 - \gamma (m + 1 - d)} \#J
    \gg Q^{n + 1}.
  \]
  However, since there are only $Q^{n+1}$ polynomials of degree at most $n$ and
  height at most $Q$, we can conclude that
  \[
    \#M_{\bm{f}}^n(Q, \gamma, \domf)
    \ll Q^{n + 1 - \gamma (m + 1 - d)}
  \]
  as well, matching the lower bound.
\end{remark}

We conclude this section by stating our extension of
Theorem~\ref{thm:independent_polynomial_khinchin}.
Define $\dualappr_{n,m+1}(\approxfdim,\approxfcodim; d)$ to be the set of
$\bm{x} \in \R^{m+1}$ such that
\[
  \max_{0 \leq k < d} \abs{P(x_k)} < \approxfdim(\polyheight{P})
  \quad \text{and} \quad
  \max_{d \leq k \leq m} \abs{P(x_k)} < \approxfcodim(\polyheight{P})
\]
for infinitely many $P \in \Z[X]$ with $\deg(P) \leq n$, and note that
$\dualappr_{n,m+1}(\approxfdim,\approxfdim; d) = \dualappr_{n,m+1}(\approxfdim)$
is the set defined in Theorem~\ref{thm:independent_polynomial_khinchin}.
When $\bm{f}$ parametrises a $d$-dimensional manifold $\manifold$ and
$f_k(\bm{x}) = x_k$ for every $0 \leq k < d$, will also write
$\dualappr_{n,\bm{f}}(\approxfdim,\approxfcodim)$ for the set of
$\bm{x} \in \domf$ such that
$\bm{f}(\bm{x}) \in \dualappr_{n,m+1}(\approxfdim,\approxfcodim; d)$.

\begin{theorem}
  \label{thm:polynomial_jarnik_divergence}
  Let $\approxfdim,\approxfcodim \colon \R^+ \to \R^+$ be decreasing functions
  such that
  \begin{equation}
    \label{eq:cnd:approxfcodim_lower_bound}
    \approxfcodim(Q) \gg \max\big\{ Q^{\frac{m-n}{m+1}}, \approxfdim(Q) \big\},
  \end{equation}
  and let $g$ be a dimension function such that $r^{-d}g(r)$ is non-increasing.
  Also assume that $r^{-\gamma}g(r)$ is increasing for some $\gamma > 0$, and
  that there are constants $r_0,\cubdiva,\cubdivb \in (0,1)$ such that
  \begin{equation}
    \label{eq:cnd:ubiquity_divergence_dimension_function}
    r^{\gamma} g(\cubdiva r) \leq \cubdivb g(r) (\cubdiva r)^{\gamma}
    \text{ for any } r \in (0,r_0).
  \end{equation}
  Further suppose that
  $\bm{f}$ is Lipschitz continuous, that $\Vand(\bm{f}) \neq 0$,
  and that $\bm{f}$ is symmetrically good of degree $n + 1 - d$ on $\domf$.
  Then
  \[
    \Haus{\dualappr_{n,\bm{f}}(\approxfdim,\approxfcodim)} = \Haus{\domf}
    \quad \text{if} \quad
    \sum_{Q = 1}^{\infty}
      Q^{n-m-1+d} \approxfcodim(Q)^{m+1-d}
      g\paren{\pointapproxf{Q}}
    = \infty.
  \]
  Moreover,
  \[
    \Leb{\dualappr_{n,\bm{f}}(\approxfdim,\approxfcodim)} = 0
    \quad \text{if} \quad
    \sum_{Q = 1}^{\infty} Q^{n-m-1} \approxfdim(Q)^d \approxfcodim(Q)^{m+1-d} < \infty.
  \]
\end{theorem}

\begin{note}
  The generalised Hausdorff measure $\HausdorffM^g$ will be introduced in
  Definition~\ref{def:generalised_hausdorff_measure}.
  For the moment observe that when $g(r) = r^d$ we have that $\HausdorffM^g$ is
  a constant multiple of the Lebesgue measure on $\R^d$, thus we recover a
  version of Theorem~\ref{thm:independent_polynomial_khinchin} for symmetrically
  good manifolds.
\end{note}

\begin{note}
  Condition~\eqref{eq:cnd:ubiquity_divergence_dimension_function} is not
  particularly restrictive, and in particular it is trivially satisfied for the
  usual Hausdorff $s$-measures, i.e. when $g(r) = r^s$ for some real $s > 0$.
\end{note}

\section{Schur polynomials}%
\label{sec:schur_polynomials}

Throughout this section, we will denote by
$\sympolyring{\tau} = \Q[T_0,\dotsc,T_{\tau-1}]^{S_{\tau}}$ the space of symmetric
polynomials in $\tau$ variables, and define
\[
  \sympolybounded{\tau}{k}
  \coloneqq \{ s \in \sympolyring{\tau} : \deg(s) \leq k \}.
\]

\begin{definition}
  \label{def:order_of_symmetric_independence}
  Let $f_0,\dotsc,f_{\tau-1}$ be a collection of $\tau$ real valued functions.
  The \emph{order of symmetric independence of $f_0,\dotsc,f_{\tau-1}$}, denoted
  by $\symord{f_0,\dotsc,f_{\tau-1}}$, is either
  \[
    \min \{k : s(f_0,\dotsc,f_{\tau-1}) \neq 0
               \text{ for every } s \in \sympolybounded{\tau}{k}\},
  \]
  or $\infty$ when $f_0,\dotsc,f_{\tau-1}$ are algebraically independent over
  $\Q$.
\end{definition}

\begin{note}
  The functions $f_0,\dotsc,f_{\tau-1}$ are algebraically independent over $\Q$
  if and only if there is no symmetric polynomial $S \in \sympolyring{\tau}$
  such that $S(f_0,\dotsc,f_{\tau-1}) = 0$. Indeed, observe that if
  $P(f_0,\dotsc,f_{\tau-1}) = 0$ for some rational polynomial $P$ in $\tau$
  variables, then
  \[
    S(T_0,\dotsc,T_{\tau-1})
    \coloneqq P^{S_{\tau}}
    = \prod_{\sigma \in S_{\tau}} P(T_{\sigma(0)},\dotsc,T_{\sigma(\tau-1)})
  \]
  is a symmetric polynomial such that $S(f_0,\dotsc,f_{\tau-1}) = 0$.
\end{note}

\begin{note}
  Comparing with Definition~\ref{def:non_symmetric}, we see that $(\bm{f}, \nu)$
  is non-symmetric of degree $k$ at $\bm{x}$ if and only if for every ball
  $B$ containing $\bm{x}$ we have
  $\symord*[\Big]{\restr{\bm{f}}{B \cap \supp\nu}} \geq k$.
\end{note}

Now consider $\lambda = (\lambda_0,\dotsc,\lambda_{\tau-1})$ with
$\lambda_0 \geq \lambda_1 \geq \dotsb \geq \lambda_{\tau-1} \geq 0$, i.e. a
partition of the integer
$\abs{\lambda} \coloneqq \lambda_0 + \dotsb + \lambda_{\tau-1}$ with
$\len{\lambda} \leq \tau$ parts. Given two such partitions $\lambda^1$ and
$\lambda^2$, we will define their sum component by component, i.e.
$\lambda^1 + \lambda^2
= (\lambda^1_0 + \lambda^2_0, \dotsc, \lambda^1_{\tau-1} + \lambda^2_{\tau-1})$.
Also, let $\mu \coloneqq (\tau-2, \tau-3, \dotsc, 0)$ be the minimal such
partition with distinct parts. Then, the \emph{algernating polynomial}
corresponding to $\lambda$ is
\[
  a_{\lambda+\mu}(T_0,\dotsc,T_{\tau-1})
  \coloneqq \det\paren*[\big]{T_i^{\lambda_j+\mu_j}}
  =
  \begin{vmatrix}
    T_0^{\lambda_0 + \mu_0}        & \cdots
      & T_0^{\lambda_{\tau-1}+\mu_{\tau-1}} \\
    \vdots                         &
      & \vdots \\
    T_{\tau-1}^{\lambda_0 + \mu_0} & \cdots
      & T_{\tau-1}^{\lambda_{\tau-1}+\mu_{\tau-1}} \\
  \end{vmatrix}.
\]

\begin{example}
  The alternating polynomial corresponding to $(0,\dotsc,0)$ is the Vandermonde
  polynomial on $\bm{T}$, i.e. $a_\mu = \Vand(\bm{T})$.
\end{example}

\begin{definition}
  \label{def:schur_polynomial}
  By Cauchy's Bi-Alternant Formula we know that $a_\mu$ divides
  $a_{\lambda + \mu}$ for every partition $\lambda$ (see \cite{Stemb2002} for a
  concise proof). Further, the quotient is a symmetric polynomial, and we define
  the \emph{Schur polynomial} in $\tau$ variables corresponding to $\lambda$ as
  \[
    s_\lambda \coloneqq \frac{a_{\lambda+\mu}}{a_\mu}.
  \]
  We can also extend this to $\len{\lambda} > \tau$ by setting $s_\lambda = 0$,
  and we will denote by $\SchurD{\tau}{k}{\bm{T}}$ the collection of all the
  $s_\lambda(\bm{T})$ with $\abs{\lambda} \leq k$ and $\ell(\lambda) \leq \tau$
  (c.f. the definition of $\Schur[n,\tau][T]$ at \eqref{eq:schur_map}).
  Note that $\#\SchurD{\tau}{k}{\bm{T}} = \binom{k + \tau - 1}{\tau}$.
\end{definition}

One can show that $s_\lambda$ is symmetric and homogeneous of degree
$\abs{\lambda}$, which makes it straighforward to see that
\[
  s_\lambda(T_0,\dotsc,T_{\len{\lambda}},0,\dotsc,0)
  = s_\lambda(T_0,\dotsc,T_{\len{\lambda}})
\]
when $\len{\lambda} < \tau$. There is a wealth of literature about Schur
polynomials, and the interested reader is invited to consult either I. G.
Macdonald's book \cite{Macdo1995}, or \cite{Egge2019} for a more gentle
introduction. In particular, we will need the following result.

\begin{proposition}[{\cite[(3.3), p.~41]{Macdo1995}}]
  \label{prop:schur_linear_basis}
  The Schur polynomials in $\SchurD{\tau}{k}{\bm{T}}$ form a basis for
  $\sympolybounded{\tau}{k}$ as a module over $\Q$.
\end{proposition}

\begin{corollary}
  \label{cor:symmetric_goodness_equivalence}
  The following are equivalent:
  \begin{itemize}
    \item $(\bm{f},\nu)$ is non-symmetric of degree $k$ on $\domf$;
    \item for every $x \in \domf \cap \supp\nu$, every neighbourhood
      $B \ni x$, and every partition $\lambda$ with $\abs{\lambda} \leq k$ and
      $\len{\lambda} \leq \tau$, the restrictions of $s_{\lambda}(\bm{f})$ to
      $B \cap \supp\nu$ are linearly independent over $\Q$.
    \item $(\Schur^k \circ \bm{f})_* \nu$ is rationally non-planar.
  \end{itemize}
\end{corollary}

\begin{remark}
  \label{rmk:friendly_measures}
  It follows that $(\bm{f},\nu)$ is symmetrically good of degree $k$ if and only
  if $(\Schur^k \circ \bm{f})_* \nu$ is decaying and rationally non-planar. C.f.
  the notion of \emph{friendly} measure from \cite{KleinLW2004}, i.e. a measure
  that is Federer, decaying and non-planar.
\end{remark}

We conclude this section with some criteria to estimate $\symord{\bm{f}}$.

\begin{proposition}
  \label{prop:symmetry_independence_upper_bound}
  Let $\bm{f} = (f_0,\dotsc,f_{\tau-1})$. Then for every $2 \leq t \leq \tau$
  and for every $I \in \indexsets{\tau}{t}$ we have
  \[
    \symord{\bm{f}} < \frac{n!}{t!}(\symord{\bm{f}_I}+1).
  \]
  \begin{proof}
    Fix $t,I$, and let $s \in \sympolyring{t}$ be a polynomial of degree
    $\symord{\bm{f}_I}+1$ such that $s(\bm{f}_I) = 0$. Since $s$ is symmetric we
    know that $S_t$ fixes $s$, thus there is a well defined action of
    $G_t \coloneqq S_\tau / S_t$ on the image of $s$ under the inclusion
    $\Q[T_0,\dotsc,T_{t-1}] \subset \Q[T_0,\dotsc,T_{\tau-1}]$. It follows that
    $s^G$ is a symmetric polynomial in $\tau$ variables of degree
    $\deg(s) \#G$ such that $s^G(\bm{f}) = 0$.
  \end{proof}
\end{proposition}

Now let $d,N$ be positive integers, and for every $0 \leq s \leq N-1$ let
$\diff{s}$ be a differential operator of the form
\begin{equation}
  \label{eq:differential_operator}
  \diff{s} = \paren{\frac{\partial}{\partial x_0}}^{j_0}
             \dotsm
             \paren{\frac{\partial}{\partial x_{d-1}}}^{j_{d-1}}
  \quad
  \text{where } j_0 + \dotsb + j_{d-1} \leq s.
\end{equation}
Given a $\cont[N-1]$ map $\bm{g}(x_0,\dotsc,x_{d-1})$ with $N-1$ components,
define the \emph{generalised Wronskian of $\bm{g}$ associated with
$\diff{0},\dotsc,\diff{N-2}$} to be the determinant
\[
  \det\paren{\diff{i}(g_j)} =
  \begin{vmatrix}
    \diff{0}(g_0)   & \dotsc & \diff{0}(g_{N-1})   \\
    \vdots          &        & \vdots              \\
    \diff{N-1}(g_0) & \dotsc & \diff{N-1}(g_{N-1})
  \end{vmatrix}.
\]
This definition can also be extended to the case where $g_0,\dotsc,g_{N-1}$ are
formal power series with coefficients in a field $K$. Furthermore, note that if
the components of $\bm{g}$ are linearly dependent, then all of its generalised
Wronskians vanish. In \cite{BostaD2010}, Bostan and Dumas proved the following
partial converse.

\begin{theorem}[{\cite[Theorem~3]{BostaD2010}}]
  \label{thm:formal_series_wronskian}
  Let $g_0,\dotsc,g_{N-1}$ be formal power series with coefficients in a field
  $K$ of charactristic $0$. If they are linearly independent over $K$, then at
  least one of their generalised Wronskians is non-zero.
\end{theorem}

\begin{corollary}[Wronskian Criterion]
  \label{cor:wronskian_criterion}
  Let $\bm{g} = (g_1,\dotsc,g_{N-1})$ be a $\cont[N-1]$ real valued map. If at
  least one of the generalised Wronskians of $\bm{g}$ is non-zero, then
  $g_1,\dotsc,g_N$ are linearly independent over $\R$, and the converse holds
  when $\bm{g}$ is analytic.
\end{corollary}

\begin{corollary}
  \label{cor:analytic_symmetric_independence}
  Let $k,\tau$ be positive integers and let $N = \binom{k + \tau - 1}{\tau}$.
  If $\bm{f}$ is a $\cont[N-1]$ real valued map with $\tau$ components and at
  least one of the generalised Wronskians of $\SchurD{\tau}{k}{\bm{f}}$ is
  non-zero, then $\symord{f} \geq k$.
\end{corollary}

To simplify the proof of the final result we will rely on another special kind
of symmetric polynomial, the \emph{monomial symmetric polynomials} $m_\lambda$.
Let $\lambda$ be a partition of integers with at most $\tau$ parts; then
$m_\lambda$ is defined as
\[
  m_\lambda
  \coloneqq \sum_\sigma T_0^{\sigma(\lambda_0)} \dotsm
                        T_{\tau-1}^{\sigma(\lambda_{\tau-1})}
\]
where $\sigma$ runs over the distinct permutations of
$\lambda_0,\dotsc,\lambda_{\tau-1}$. Again, it can be shown that the collection
of monomial symmetric polynomials corresponding to $\lambda$ with
$\abs{\lambda} \leq k$ and $\len{\lambda} = \tau$ forms a basis for
$\sympolybounded{\tau}{k}$ as a module over $\Q$.

\begin{proposition}
  \label{prop:minors_polynomial}
  Let $\bm{p} = (p_0, p_1)$ be a polynomial map with $\deg p_0 > \deg p_1$. Then
  $\symord{\bm{p}} \geq \frac{\deg p_0}{\deg p_1}$.
  \begin{proof}
    Let $d_i \coloneqq \deg p_i$, and note that if
    $\lambda = (\lambda_0,\lambda_1)$ is a partition with
    $k \geq \lambda_0 \geq \lambda_1 \geq 0$, then
    $\deg m_\lambda(\bm{p}) = d_0 \lambda_0 + d_1 \lambda_1$. We will show that
    the map $\lambda \mapsto \deg m_\lambda(\bm{p})$ is injective for
    $k \leq \frac{d_0}{d_1}$, which immediately gives a lower bound for
    $\symord{\bm{p}}$.

    Suppose that
    $d_0 \lambda^1_0 + d_1 \lambda^1_1 = d_0 \lambda^2_0 + d_1 \lambda^2_1$ for
    some $\lambda^1 \neq \lambda^2$, and without loss of generality assume
    $\lambda^1_1 > \lambda^2_1$. Then
    $d_1 (\lambda^1_1 - \lambda^2_1) = d_0 (\lambda^2_0 - \lambda^1_0)$, which
    results in
    \begin{align*}
      k &\geq \lambda^2_0 \\
        &= \lambda^1_0 + \frac{d_1}{d_0} (\lambda^1_1 - \lambda^2_1) \\
        &\geq 1 + \frac{d_1}{d_0} \qedhere
    \end{align*}
  \end{proof}
\end{proposition}

\begin{example}
  At least when $\bm{p}(\bm{x}) = (p_0(\bm{x}),\dotsc,p_{\tau-1}(\bm{x}))$
  is a polynomial map with rational coefficients, we can compute
  $\symord{\bm{p}}$ with relative efficiency using variable elimination via
  Gr\"{o}bner bases. Even more, it is possible to describe all the symmetric
  polynomials that vanish on $\bm{p}$. Indeed, let $e_1,\dotsc,e_\tau$ be the
  elementary symmetric polynomials in $\tau$ variables, that is
  \[
    e_k(T_0,\dotsc,T_{\tau-1})
      \coloneqq \sum_{I \in \indexsets{\tau-1}{k}} T_{i_1} \dotsm T_{i_k}.
  \]
  Then it is well known that
  $\sympolyring{\tau} = \Q[e_1(\bm{T}),\dotsc,e_\tau(\bm{T})]$; in other words,
  every symmetric polynomial in $\bm{T}$ can be written as a polynomial in
  $e_1,\dotsc,e_\tau$. Now let $\bm{Y} = (Y_1,\dotsc,Y_\tau)$ and consider the ideal
  $\mathcal{I} \subset \Q[\bm{x},\bm{Y}]$ generated by the polynomials
  \[
    Y_k - e_k(\bm{p})
    \qquad \text{for }
    1 \leq k \leq \tau.
  \]
  It is possible to compute a Gr\"{o}bner basis $G$ for the ideal
  $\tilde{\mathcal{I}} \coloneqq \mathcal{I} \cap \Q[\bm{Y}]$ through standard
  algorithms, and we can see that every symmetric polynomial in
  $\sympolyring{\tau}$ which vanishes on $\bm{p}$ is of the form
  $h(e_1,\dotsc,e_\tau)$ for some $h \in \tilde{\mathcal{I}}$. In particular,
  \[
    \symord{\bm{p}} = \min_{g \in G} \deg(g(e_1,\dotsc,e_\tau)).
  \]
  As an example, these are the orders of symmetric independence for the Veronese
  curves of degree $\tau$ between $2$ and $10$, i.e. for
  $\bm{p}(x) = (x,x^2,\dotsc,x^\tau)$:
  \[
    \begin{array}{c | c c c c c c c c c}
      \tau            & 2 & 3 & 4 & 5 & 6 & 7 & 8 & 9 & 10 \\
      \hline
      \symord{\bm{p}} & 4 & 5 & 5 & 6 & 6 & 7 & 7 & 7 &  7
    \end{array}
  \]
\end{example}

\section{Good functions}
\label{sec:good_functions}

\begin{definition}
  \label{def:good_function}
  Let $X$ be a metric space and $\nu$ a Radon measure on $X$. Also consider an
  open subset $U \subseteq X$ and a $\nu$-measurable function
  $f \colon U \to \R$.
  For any open ball $B \subset U$ and $\varepsilon > 0$, define
  \[
    B^{f,\varepsilon} \defeq \{ x \in B : \abs{f(x)} < \varepsilon \}.
  \]
  Then we say that $f$ is \emph{$(C,\alpha)$-good on $U$ with respect to $\nu$} if
  there are constants $C,\alpha > 0$ such that for any open ball $B \subset U$
  centred on $\supp \nu$ we have
  \begin{equation}
    \label{eq:cnd:goodness}
    \nu\left(B^{f,\varepsilon}\right) \leq
    C \left( \frac{\varepsilon}{\norm{f}_{\nu,B}} \right)^\alpha \nu(B)
    \quad \text{for all } \varepsilon > 0,
  \end{equation}
  where $\norm{f}_{\nu,B} \defeq \sup_{x \in B \cap \supp\nu} \abs{f(x)}$. Also,
  when $X = \R^d$ and $\nu$ is the corresponding Lebesgue measure we will write
  $\norm{f}_B$ for $\norm{f}_{\nu,B}$, and we say that $f$ is
  \emph{absolutely $(C,\alpha)$-good on $U$ with respect to $\nu$} if
  \eqref{eq:cnd:goodness} holds wiht $\norm{f}_B$ in place of
  $\norm{f}_{\nu,B}$.
\end{definition}

Note that absolute $(C,\alpha)$-goodness implies
$(C,\alpha)$-goodness, while the converse holds for measures with full support.
The following properties are a direct consequence of the definition.

\begin{lemma}[{\cite[Lemma~3.1]{KleinT2007}}, {\cite[Lemma~3.1]{BerniKM2001}}]
  \label{lm:good_function_properties}
  \leavevmode
  \begin{enumerate}
    \item $f$ is $(C,\alpha)$-good on $U$ wrt $\nu$ if and only if so is
      $\abs{f}$.
    \item If $f$ is $(C,\alpha)$-good on $U$ wrt $\nu$, then so is $\lambda f$ for
      every $\lambda \in \R$.
    \item If $f$ is $(C,\alpha)$-good on $U$ wrt $\nu$, then it is also
      $(C',\alpha')$-good on $U'$ wrt $\nu$ for every $C' \geq C$,
      $\alpha' \leq \alpha$ and $U' \subseteq U$.
    \item If $\{f_i\}_{i \in I}$ is a collection of $(C,\alpha)$-good functions
      on $U$ wrt $\nu$ and the function $f \defeq \sup_{i \in I} \abs{f_i}$ is
      Borel measurable, then $f$ is also $(C,\alpha)$-good on $U$ wrt $\nu$.
    \item \label{lm:good_function_properties:ratio}
      If $f$ is $(C,\alpha)$-good on $U$ wrt $\nu$ and
      $\detmin \leq \frac{\abs{f(x)}}{\abs{g(x)}} \leq \detmax$ for every
      $x \in U \cap \supp \nu$, then $g$ is
      $(C(\detmax/\detmin)^\alpha, \alpha)$-good on $U$ wrt $\nu$.
  \end{enumerate}
  \begin{proof}[Proof of \ref{lm:good_function_properties:ratio}]
    Note that if $\varepsilon > \abs{g(x)} \geq \frac{\abs{f(x)}}{\detmax}$ on
    $U \cap \supp \nu$, then
    $B^{g,\varepsilon} \cap \supp \nu \subseteq
     B^{f,\detmax\varepsilon} \cap \supp \nu$ for every ball $B \subseteq U$.
    Furthermore,
    \[
      \detmin \norm{g(x)}_{\nu,B}
      = \sup_{x \in B \cap \supp \nu} \detmin\abs{g(x)}
      \leq \sup_{x \in B \cap \supp \nu} \abs{f(x)}
      = \norm{f(x)}_{\nu,B}.
    \]
    Therefore
    \begin{align*}
      \nu\left(B^{g,\varepsilon}\right)
      &\leq \nu\left(B^{f,\detmax\varepsilon}\right) \\
      &<    C \left( \frac{\detmax\varepsilon}{\norm{f}_{\nu,B}}
              \right)^\alpha \nu(B) \\
      &\leq C \left(\frac{\detmax}{\detmin}\right)^\alpha
              \left( \frac{\varepsilon}{\norm{g}_{\nu,B}}
              \right)^\alpha \nu(B).
              \qedhere
    \end{align*}
  \end{proof}
\end{lemma}

The papers \cite{KleinM1998} and \cite{BerniKM2001} include various examples of
real valued functions which are $(C,\alpha)$-good with respect to Lebesgue
measure. Moreover, \cite{KleinT2007} extends those examples to functions with
values in non-Archimedean fields which satisfy a condition equivalent to
\eqref{eq:cnd:goodness}. For the purposes of the present paper we are
mainly interested in the following propositions.

\begin{proposition}[{\cite[Proposition~2.8]{AkaBRS2015}}]
  \label{prop:good_polynomials}
  Fix $d,m,k \in \Z_{>0}$ and let
  $\bm{g} = (g_1, \dotsc, g_N) \colon \R^d \to \R^N$
  be a polynomial map of degree at most $k$. Then for any convex subset
  $B \subset \R^d$ we have
  \[
    \Leb{\left\{x \in B : \norm{\bm{g}(\bm{x})} < \varepsilon \right\}}
    \leq 4 d \left( \frac{\varepsilon}{\norm{\bm{g}}_B} \right)^{\frac{1}{k}}
    \Leb{B},
  \]
  where $\norm{\bm{g}}_B = \sup_B \norm{\bm{g}(\bm{x})}$ and
  $\norm{\bm{g}(\bm{x})} = \max_j \abs{g_j(\bm{x})}$.
\end{proposition}

\begin{note}
  This immediately implies that polynomial functions on $\R^d$ of degree $k$ are
  $(4d,1/k)$-good with respect to Lebesgue measure.
\end{note}

Now, suppose that $U \subset \R^d$ is open and that
$\bm{g} = (g_1,\dotsc,g_N) \colon U \to \R^N$ is a $\cont[\ell]$ map. For a
given $\bm{x} \in U$, we say that $\bm{g}$ is \emph{$\ell$-nondegenerate at
$\bm{x}$} if the partial derivatives of $\bm{g}$ at $\bm{x}$ of order up to
$\ell$ span $\R^N$. In \cite{KleinM1998} Kleinbock and Margulis proved the
following result on the $(C,\alpha)$-goodness with respect to Lebesgue measure
of $\ell$-nondegenerate functions, which was later extended in \cite{KleinLW2004}
to a wider class of measures.

\begin{proposition}[{\cite[Proposition~3.4]{KleinM1998}}]
  \label{prop:good_linear_combinations_lebesgue}
  Let $\bm{g} = (g_1,\dotsc,g_N) \colon U \subseteq \R^d \to \R^N$ be a
  $\cont[\ell]$ map, $U$ open. If $\bm{g}$ is $\ell$-nondegenerate at
  $\bm{x} \in U$, then there are a neighbourhood $V \subset U$ of $\bm{x}$ and a
  $C > 0$ such that any linear combination of $1,g_1,\dotsc,g_N$ is
  $(C, 1/d\ell)$-good on $V$ with respect to Lebesgue measure.
\end{proposition}

Recall from Definition~\ref{def:measure_properties} that a measure $\nu$ on $X$
is called \emph{Federer} if for $\nu$ almost every $x \in X$ there are a
neighbourhood $U$ of $x$ and a constant $D > 0$ such that
$\nu(3^{-1}B) > \nu(B)/D$ for any ball $B \subset U$ centred on $\supp\nu$.
Furthermore, observe that $\nu$ is absolutely $(C,\alpha)$-decaying according to
\eqref{eq:decaying_measure} precisely when every linear function is absolutely
$(C,\alpha)$-good with respect to $\nu$.

\begin{proposition}[{\cite[Proposition~7.3]{KleinLW2004}}]
  \label{prop:good_linear_combinations_general}
  Let $\bm{g} = (g_1,\dotsc,g_N) \colon U \subseteq \R^d \to \R^N$ be a
  $\cont[\ell+1]$ map, $U$ open. Further, let $\nu$ be a measure which is
  Federer and absolutely $(\tilde{C},\alpha)$-decaying on $U$ for some
  $\tilde{C},\alpha > 0$. If $\bm{g}$ is $\ell$-nondegenerate at
  $\bm{x} \in U$, then there are a neighbourhood $V \subset U$ of $\bm{x}$ and a
  $C > 0$ such that any linear combination of $1,g_1,\dotsc,g_N$ is absolutely
  $\left(C, \alpha/(2^{\ell+1}-2)\right)$-good on $V$ with respect to $\nu$.
\end{proposition}

\begin{note}
  Consider the Lebesgue measure as $\nu$. Then
  Proposition~\ref{prop:good_polynomials} shows that for $d=1$ the exponent $1/\ell$
  in Proposition~\ref{prop:good_linear_combinations_lebesgue} is likely to be
  optimal, while $1/(2^{\ell+1}-2)$ is much worse. However, the latter is
  independent of $d$. Unfortunately, according to \cite{KleinLW2004}, finding
  the optimal exponent seems to be a challenging open problem.
\end{note}

\begin{corollary}
  \label{cor:analytic_nondegenerate}
  Let $k$ be a positive integer, $\bm{f} = (f_0,\dotsc,f_{\tau-1})$ be an
  analytic map on $U \subset \R^d$, and let $\nu$ be a measure on $U$ which is
  Federer and absolutely $(\tilde{C},\tilde{\alpha})$-decaying on $U$. Then for
  every $\bm{x} \in U \setminus Z_{\bm{f}}$ there are a neighbourhood
  $V \ni \bm{x}$ and constants $C_{\bm{x}},\alpha > 0$ such that $s(\bm{f})$ is
  $(C_{\bm{x}},\alpha)$-good on $V$ for every symmetric polynomial $s$ of degree
  up to $k$, where $Z_{\bm{f}}$ is the zero set of a real anlytic function.
  Furthermore, if $N = \binom{k+\tau-1}{\tau}$, then $\alpha$ can be chosen to
  be:
  \begin{itemize}
    \item $\tilde{\alpha}/(2^N - 2)$;
    \item $1/d(N-1)$ if $\nu$ is the Lebesgue measure;
    \item $1/k\deg(\bm{f})$ if $\bm{f}$ is a polynomial map and $\nu$ is the
      Lebesgue measure.
  \end{itemize}
  \begin{proof}
    Let $\Sigma$ be a basis for the linear span
    $\linspan[\R]{\SchurD{\tau}{k}{\bm{f}}}$, and note that we may always assume
    that $1 \in \Sigma$, since $1 \in \SchurD{\tau}{k}{\bm{f}}$ for every
    $k > 0$. Furthermore, $\# \Sigma \leq N$ and all the elements of $\Sigma$
    are analytic because so is $\bm{f}$.

    Therefore by the Wronskian Criterion (Corollary~\ref{cor:wronskian_criterion})
    we know that at least one of the generalised Wronskians of $\Sigma$, say
    $W$, is not identically zero. Hence $\Sigma$ is non-degenerate outside of
    the zero set $Z_{\bm{f}}$ of $W$, and the statement follows from
    Propositions \ref{prop:good_polynomials},
    \ref{prop:good_linear_combinations_lebesgue} and
    \ref{prop:good_linear_combinations_general}.
  \end{proof}
\end{corollary}

\begin{remark}
  \label{rmk:zero_set_of_an_analytic_function}
  As a consequence of the Implicit Function Theorem, one can show that the
  Hausdorff dimension of $Z_{\bm{f}}$ is at most $d - 1$ \cite{Mitya2020}. In
  particular, if $\nu$ is either the Lebesgue measure on $\R^d$ or the natural
  measure supported on a sufficiently regular IFS of dimension $s > d-1$ (see
  Example~\ref{ex:federer_and_decaying_measures}), then $\nu$ satisfies the
  hypotheses of Corollary~\ref{cor:analytic_nondegenerate} and
  $\nu(Z_{\bm{f}}) = 0$.
\end{remark}

\section{Points with conjugate coordinates}
\label{sec:conjugate_coordinates}

\begin{proof}[Proof of Theorem~\ref{thm:lower_bound_for_manifold_count}
              assuming Corollary~\ref{cor:n+1_points}]
  $ $\newline
  First, we are now going to check that the hypotheses of
  Corollary~\ref{cor:n+1_points} are satisfied for $\psi$ and $\Psi$ defined
  by
  \[
    \begin{cases}
      \psi_k^d(Q) = \psi^d(Q)
        \coloneqq Q^{d - n - 1 + \gamma (m + 1 - d)}
        & \text{for } 0 \leq k < d \\
      \psi_k(Q) = \Psi(Q) \defeq Q^{1-\gamma}
        & \text{for } d \leq k \leq m.
    \end{cases}
  \]
  Clearly $\gamma > 0$ implies that $\psi_k(Q)Q^{-1} = Q^{-\gamma}$ is
  decreasing, and observe that
  \[
    \bm{\psi}(Q)Q^{n-m} = \psi(Q)^d \Psi(Q)^{m+1-d} Q^{n-m}
                        = 1,
  \]
  hence
  condition~\eqref{eq:cnd:t_positive_d} of
  Theorem~\ref{thm:bound_over_an_interval} is satisfied for every
  $1 \leq \tau \leq m+1$ by choosing $I_{\tau} = (0,\dotsc,\tau-1)$, since our
  choice of $\psi$ and $\Psi$, together with the fact that $\bm{\psi}$ is
  decreasing, implies that $\bm{\psi}_{I_\tau}$ is decreasing as well.

  Furthermore, observe that $\symord{\bm{x}} = \infty$ because the coordinate
  functions $x_0,\dotsc,x_{d-1}$ are algebraically independent over $\R$.
  Therefore \eqref{eq:cnd:count_bound} is enough to guarantee that
  condition~\eqref{eq:cnd:minors_d} is satisfied as well, since the
  ordering of $\psi_k$, hence of $f_k$, is irrelevant for $d \leq k \leq m$.

  Thus we can apply Corollary~\ref{cor:n+1_points} taking the Lebesgue measure
  $\vol_d$ on $\R^d$ as $\nu$, and through it we find
  $\countdomgood{\theta} \subseteq \domf$ with
  $\vol_d(\countdomgood{\theta}) \gg \vol_d(\domf)$.
  Moreover, for every $x \in \countdomgood{\theta}$ we have points
  $(\alpha_0, \dotsc, \alpha_m)$ with algebraic conjugate coordinates and
  $\polyheight{\alpha_k} \ll Q$ such that
  \begin{equation}
    \label{eq:proof_n+1_points_ball}
    \begin{cases}
      \abs{x_k - \alpha_k} \ll \frac{\psi(Q)}{Q}
        & \text{for } 0 \leq k < d \\
      \abs{f_k(\bm{x}) - \alpha_k} \ll \frac{\Psi(Q)}{Q}
        & \text{for } d \leq k \leq m.
    \end{cases}
  \end{equation}
  However, given that $\psi$ and $\Psi$ are multiplicative, we may assume that
  $\polyheight{\alpha_k} \leq Q$ by rescaling $Q$ and changing the implied
  constants in \eqref{eq:proof_n+1_points_ball} accordingly.

  Now choose a compact subset $K \subseteq \countdomgood{\theta}$ such that
  $\vol_d(K) \gg \vol_d(\domf)$, which we can always do since $\domf$ is
  assumed to be bounded. Then note that the partial derivatives of $\bm{f}$ are
  all bounded on $K$, therefore the Mean Value Theorem implies that for each
  $d \leq k \leq m$ we have
  \[
    \abs{f_k(\bm{\alpha}) - f_k(\bm{x})}
    \ll_{K,d,f_k} \max_{0 \leq i < d} \abs{x_i - \alpha_i}
    \ll \frac{\psi(Q)}{Q}.
  \]
  It follows that
  \[
    \abs{f_k(\bm{\alpha}) - \alpha_k}
    \leq \abs{f_k(\bm{\alpha}) - f_k(\bm{x})} + \abs{f_k(\bm{x}) - \alpha_k}
    \ll \frac{\psi(Q)}{Q} + \frac{\Psi(Q)}{Q}.
  \]
  Since $\Psi(Q) \geq \psi(Q)$ precisely when
  $\gamma \leq \frac{n + 1}{m + 1}$, we have that
  \begin{equation}
    \label{eq:proof_n+1_bound}
    \abs{f_k(\bm{\alpha}) - \alpha_k} \ll \frac{\Psi(Q)}{Q} = Q^{-\gamma},
  \end{equation}
  hence if $\capprcount$ is greater than the implied constant, then
  $(\alpha_0,\dotsc,\alpha_m) \in M_{\bm{f}}^n(Q, \gamma, \domf)$.
  Therefore we conclude that
  \begin{align*}
    \#M_{\bm{f}}^n(Q, \gamma, \domf)
    &\gg \vol_d(K) \left(\frac{Q}{\psi(Q)}\right)^d \\
    &\gg \vol_d(\domf) Q^{n + 1 - \gamma (m + 1 - d)}
    \qedhere.
  \end{align*}
  \end{proof}

In section~\ref{sec:tailored_polynomials} we shall prove the following Corollary
of Theorem~\ref{thm:bound_over_an_interval}, from which
Corollary~\ref{cor:n+1_points} follows immediately.

\begin{corollary}
  \label{cor:tailored_polynomials}
  Consider $\psi_0,\dotsc,\psi_{m}$, $\varphi_{m+2}, \dotsc, \varphi_n$,
  $\bm{f}$ as in Theorem~\ref{thm:bound_over_an_interval} such that
  condition~\eqref{eq:cnd:t_positive_d} holds with
  \[
    \varphi(Q) =  \varphi_{m+1}(Q)
    = \max\left\{ Q, \varphi_{m+2}(Q), \dotsc, \varphi_n(Q) \right\}.
  \]
  Furthermore, assume that there are constants $\deltamin, \deltamax > 0$ such
  that
  \begin{equation}
    \label{eq:cnd:volume_constant}
    \deltamin \leq \bm{\psi}(Q) \bm{\varphi}(Q) \leq \deltamax.
  \end{equation}
  Then for every $0 < \theta < 1$ there are constants
  $\ctailoredy,\ctailoredf > 0$ and a subset $\countdomgood{\theta}$ of $\domf$,
  independent of $Q$, such that $\nu(\countdomgood{\theta}) > \theta \nu(\domf)$
  and every $\bm{x} \in \countdomgood{\theta}$ admits $n+1$ linearly independent
  irreducible polynbomials $P = a_0 + a_1 X + \dotsb + a_n X^n \in \Z[X]$ of
  degree bounded by $n$ such that
  \begin{equation}
    \label{eq:tailored_polynomial_system}
      \begin{cases}
        \begin{array}{@{}c@{}}
          \abs{P\left(f_k(\bm{x})\right)} < \ctailoredy \psi_k(Q) \\
          \abs{P'\left(f_k(\bm{x})\right)} > \ctailoredf \varphi(Q)
        \end{array}
          & \text{for } 0 \leq k \leq m \\
        \abs{a_k} \leq \ctailoredy \varphi_k(Q)
          & \text{for } m < k \leq n
      \end{cases}
  \end{equation}
  whenever $Q$ is sufficiently large. In particular
  $\polyheight{P} \ll \varphi(Q)$.
\end{corollary}

\begin{proof}[Proof of Corollary~\ref{cor:n+1_points}]
  Let $P$ be as in the statement of Corollary~\ref{cor:tailored_polynomials},
  let $\varphi_{m+1}(Q) = \dotsb = \varphi_n(Q) = Q$, so that $\varphi(Q) = Q$
  as well, and note that by remark~\ref{rmk:compactness_of_domfgood} we may
  choose $\domfgood{\theta}$ to be compact.
  To simplify the notation, let $y_k = f_k(\bm{x})$ for $0 \leq k \leq m$. Then
  observe that, since $P'$ is continuous, $\domfgood{\theta}$ is compact, and
  $\psi_k(Q)Q^{-1}$ is decreasing, we may choose an open set $U$ with
  $\domfgood{\theta} \subset U \subseteq \domf$ and a constant $Q_0 > 0$ such
  that for every $Q > Q_0$ every interval of the form
  \[
    I_{y_k} \defeq \left[y_k - \kappa \frac{\psi_k(Q)}{Q},
                         y_k + \kappa \frac{\psi_k(Q)}{Q}\right]
  \]
  is contained in $U$, where $\kappa \coloneqq \frac{\ctailoredy}{\ctailoredf}$,
  and such that $\abs{P'(z)} > \ctailoredf Q$ for every $z \in U$.
  Furthermore, by the Mean Value Theorem we know that for every
  $\tilde{y}_k \in I_{y_k}$ there is a $z_k \in I_{y_k}$ such that
  \begin{align*}
    P(\tilde{y}_k)
    &= P(y_k) + P'(z_k) (\tilde{y}_k - y_k).
  \end{align*}
  Now note that $\polyheight{P} \ll Q$, again because
  $\psi_k(Q)Q^{-1}$ is decreasing for every $0 \leq k \leq m$. As
  $\domf$ is bounded, it follows that $\abs{P'(z_k)}$ is bouded above by
  $Q$, up to a constant that depends on $n$, $\bm{f}$ and $\domf$.
  Furthermore, $\abs{P'(z_k)} > \ctailoredf Q$ implies that for
  $\tilde{y}_k = y_k \pm \kappa \frac{\psi_k(Q)}{Q}$ we have
  \[
    \abs{P'(z_k) (\tilde{y}_k - y_k)}
    > \ctailoredf \kappa \psi_k(Q)
    = \ctailoredy \psi_k(Q),
  \]
  therefore
  \[
    P\left(y_k - \kappa \frac{\psi_k(Q)}{Q}\right)
    P\left(y_k + \kappa \frac{\psi_k(Q)}{Q}\right) < 0.
  \]

  Applying once more the Mean Value Theorem we obtain, for every
  $0 \leq k \leq m$, a root $\alpha_k$ of $P$ such that
  \[
    \abs{y_k - \alpha_k} < \kappa \frac{\psi_k(Q)}{Q}.
  \]
  Finally, note that Corollary~\ref{cor:tailored_polynomials} gives us $n+1$
  distinct irreducible polynomials, from which we obtain $n+1$ distinct points
  $(\alpha_0, \dotsc, \alpha_m)$.
\end{proof}

\begin{note}
  The numbers $y_k$ are pairwise distinct on $\domf \cap \supp\nu$, since
  $\det U^h_{\bm{f}}$ is non-zero by remark~\ref{rmk:non_zero_determinant}.
  By taking $Q$ large enough if necessary, it follows that we can guarantee
  that the sets $I_{y_k}$ are pairwise disjoint, hence the roots $\alpha_k$ are
  pairwise distinct. In particular, observe that the constant $\kappa$ does not
  depend on $Q$ and may be chosen uniformly on $\domf$.
\end{note}

\section{Tailored polynomials}
\label{sec:tailored_polynomials}


Similarly to what Beresnevich, Bernik and G\"otze did in \cite{BeresBG2010}, we
call a \emph{tailored polynomial} an irreducible polynomial which satisfies
\eqref{eq:d_dimensional_inequalities}. Our construction follows closely the
argument of \cite[Section~3]{BeresBG2010}, and it is based on
Theorem~\ref{thm:bound_over_an_interval}, which we will then prove in
Section~\ref{sec:proof_of_the_main_theorem} using the quantitative non-divergence
method of Kleinbock and Margulis.

Now fix $\bm{x} \in \domf$ and observe that solving for $P \in \Z[X]$ the system
of inequalities
\begin{equation}
  \label{eq:starting_minkowski}
  \begin{cases}
    \abs{P(f_k(\bm{x}))} < \psi_k(Q)
      & \text{for } 0 \leq k \leq m \\
    \abs{a_k} \leq \varphi_k(Q)
      & \text{for } m < k \leq n
  \end{cases}
\end{equation}
is equivalent to looking for points of the lattice $L \defeq M\Z^{n+1}$ which
lie in the convex body $\convexbody$, where $M = M_{\bm{f}}(\bm{x})$ is the matrix
defined in \eqref{eq:main_matrix} and where
\[
  \convexbody \defeq
  \left\{ \bm{y} \in \R^{n+1} :
    \begin{array}{lr}
      \abs{y_k} < \psi_k(Q) & \text{ for } 0 \leq k \leq m \\
      \abs{y_k} \leq \varphi_k(Q) & \text{ for } m < k \leq n
    \end{array} \right\}.
\]
Note that $\det M \neq 0$ on $\domf \cap \supp \nu$, since
$\Vand(\bm{f}) \neq 0$ implies $\det U^h_{\bm{f}} \neq 0$ by
remark~\ref{rmk:non_zero_determinant}. Furthermore, since $\det M$ is
continuous in $\bm{x}$ we may assume without loss of generality that is bounded
away from $0$ on $\domf$, up to replacing $\domf$ with the interior of a compact
subset with measure arbitrarily close to $\nu(\domf)$ (which we can always find
since $\nu$ is Radon).
Then Minkowski's second convex body theorem tells us that the successive minima
$\lambda_0 \leq \dotsc \leq \lambda_n$ of $\convexbody$ with respect to $L$
satisfy
\[
  \frac{2^{n+1}}{(n+1)!} \det M
  \leq \lambda_0 \dotsm \lambda_n \vol(\convexbody)
  \leq 2^{n+1} \det M
\]
where $\vol(\convexbody) = 2^{n+1}\bm{\psi}(Q)\bm{\varphi}(Q)$ is the volume of
$\convexbody$. Therefore we have
\[
  \lambda_n \leq \frac{\det(M)}{\deltamin \lambda_0^n},
\]
since $\bm{\psi}(Q)\bm{\varphi}(Q) \geq \deltamin$ by
condition~\eqref{eq:cnd:volume_constant}.

Now note that if $P = a_0 + a_1 X + \dotsb + a_n X^n$ is such that
$M\bm{a} \in \lambda_0 \convexbody$ where
$\bm{a} = (a_0,\dotsc,a_n)^T \neq (0,\dotsc,0)^T$, then
$\polyheight{P} \ll \lambda_0 \varphi(Q)$ as long as $\det(M)$ is uniformly
bounded away from $0$. Indeed, there is a $\bm{b} \in \lambda_0\convexbody$ such
that $M\bm{a} = \bm{b}$, thus for $Q$ large enough
\[
  \polyheight{P} = \norm{\bm{a}}_\infty
  \leq \norm{M^{-1}}_\infty \norm{\bm{b}}_\infty
  \leq \lambda_0 \varphi(Q) \frac{\norm{\adj(M)}_\infty}{\abs{\det(M)}}
\]
where $\adj(M)$ is the adjugate matrix of $M$, whose norm depends only on $n$,
$\bm{x}$ and $\bm{f}(\bm{x})$, and thus can be bounded above by a constant
depending on $n$, $\bm{f}$, and $\domf$. Since $\domf$ is bounded, it follows
that there is a constant $\cdermax > 0$ such that
\[
  \max_{0 \leq i \leq m} \abs{P'(f_i(\bm{x}))}
  \leq \cdermax \lambda_0 \varphi(Q).
\]
Therefore Theorem~\ref{thm:bound_over_an_interval} implies that
for any given $\delta_0 > 0$ the set of $\bm{x} \in \domf$ for which
$\lambda_0 = \lambda_0(\bm{x}) \leq \delta_0$ is bounded above by
\[
  \delta_0^\alpha \nu(\domf)
\]
up to a constant, since condition~\eqref{eq:cnd:volume_constant}
implies that
$\vol(\lambda_0 \convexbody) \leq 2^{n+1} \deltamax \lambda_0^{n+1}$.
In particular, we may choose $\delta_0$ depending only on $\theta$, $n$,
$\bm{f}$ and $\domf$ such that for every $\bm{x}$ in a subset $\domtp(\delta_0)$
of measure at least $\sqrt{\theta}\nu(\domf)$ we have $\lambda_0 > \delta_0$.

Now, let $\delta_n \defeq \frac{\det(M)}{\deltamin \delta_0^n}$. Then for any
$\bm{x} \in B(\delta_0)$ we may find $n+1$ linearly independent polynomials
$P_i$ whose vectors of coefficients $\bm{a}_i$ satisfy
$M\bm{a}_i \in \delta_n\convexbody$. If $A$ is the matrix with columns
$\bm{a}_i$, $0 \leq i \leq n$, then
\[
  1 \leq \abs{\det(A)} \leq \vol(\delta_n \convexbody)
                       \leq 2^{n+1} \deltamax \delta_n^{n+1} \defeq c'
\]
and by Bertrand's postulate we may find a prime $p$ such that
\[
  c' < p < 2 c'.
\]
In particular, this implies that $\det(A) \neq 0 \pmod{p}$, hence the system
\[
  A \bm{t} \equiv \bm{b}
\]
has a unique solution $\bm{t} \in \Fp^{n+1}$, where
$\bm{b} = (0,\dotsc,0,1)^T$. Now, for $\ell = 0,\dotsc,n$ define
$\bm{r}_\ell = (1,\dotsc,1,0,\dotsc,0)^T \in \Fp^{n+1}$, where $\ell$ denotes
the number of zeroes. Then write $A\bm{t} - \bm{b} = p\bm{w}$ after choosing
representatives for $\bm{t}$ in $\{0,\dotsc,p-1\}$,
let $\bm{\gamma}_\ell \in \Fp^{n+1}$ be the unique solution to
\[
  A\bm{\gamma}_\ell \equiv - \bm{w} + \bm{r}_\ell
\]
modulo $p$, and define $\bm{\eta}_\ell = \bm{t} + p \bm{\gamma}_\ell$. For each
$\ell = 0,\dotsc,n$ let
\[
  \widetilde{P}_\ell \defeq \sum_{i = 0}^n \eta_{\ell i} P_i
\]
and note that the linear independence of the vectors $\bm{r}_\ell$ implies the
linear independence of the polynomials $\widetilde{P}_\ell$.

Since $A\bm{\eta}_\ell = \bm{s}$ is the vector of coefficients of
$\widetilde{P}_\ell$ and since $\bm{\eta}_\ell \equiv \bm{t} \pmod{p}$, it
follows that $s_n \equiv 1 \pmod{p}$ and $s_i \equiv 0 \pmod{p}$ for
$0 \leq i \leq n-1$. Furthermore, the definition of $\bm{\gamma}_\ell$ implies
that
\[
  A\bm{\eta}_\ell = \bm{b} + p\bm{r}_\ell,
\]
thus $s_0 \equiv p \pmod{p^2}$. Therefore, by Eisenstein's criterion it follows
that $\widetilde{P}_\ell$ is irreducible. Finally, observe that taking
representatives for $\bm{t}$ and $\bm{\gamma}_\ell$ in $\{0,\dotsc,p-1\}$ we
have $\abs{\eta}_{\ell i} \leq p^2$, thus $\widetilde{P}_\ell$ satisfies
\begin{equation}
  \label{eq:irreducible_minkowski}
  \begin{cases}
    \abs{P(f_k(\bm{x}))} < \ctailoredy \psi_k(Q)
      & \text{for } 0 \leq k \leq m \\
    \abs{a_k} \leq \ctailoredy \varphi_k(Q)
      & \text{for } m < k \leq n,
  \end{cases}
\end{equation}
where
\begin{equation}
  \label{eq:ctailoredy}
  \begin{aligned}
    \ctailoredy
      &= 4(n+1)\delta_nc'^2 \\
      &= 2^{2n+4} (n+1) \deltamax^2 \delta_n^{2n+3} \\
      &= 2^{2n+4} (n+1) \deltamax^2
                        \paren{\frac{\det(M)}{\deltamin \delta_0^n}}^{2n+3}.
  \end{aligned}
\end{equation}
Then, Theorem~\ref{thm:bound_over_an_interval} implies that the measure of the
set of $\bm{x} \in \domf$ which admit a solution $P$ to
\eqref{eq:irreducible_minkowski} such that
$\max \abs{P'(f_i(\bm{x}))} \leq \ctailoredf \varphi(Q)$ is bounded above by
\[
  \ctailoredf^{\frac{\alpha}{n+1}}
  \ctailoredy^{\frac{\alpha n}{n+1}}
  \nu(\domf)
\]
up to a constant. In particular, we may choose $\ctailoredf > 0$, depending only
on $\theta$, $n$, $\bm{f}$ and $\domf$, such that for every $\bm{x}$ in a subset
$\countdomgood{\theta} = \domtp(\delta_0, \ctailoredf)
 \subseteq \domtp(\delta_0)$
 of measure at least
$\sqrt{\theta} \nu\paren{ \domtp(\delta_0) } \geq \theta \nu(\domf)$ we
have $\min \abs{P'(f_k(\bm{x}))} > \ctailoredf \varphi(Q)$.

\section{Ubiquity}%
\label{sec:ubiquity}

\begin{definition}
  \label{def:generalised_hausdorff_measure}
  A \emph{dimension function} $g \colon \R^+ \to \R^+$ is a continuous
  increasing function such that $g(r) \to 0$ as $r \to 0$. Now suppose that $F$
  is a non-empty subset of a metric space $\Omega$. For $\rho > 0$, a
  \emph{$\rho$-cover} of $F$ is a countable collection $\{B_i\}$ of balls in
  $\Omega$ of radii $r(B_i) \leq \rho$ whose union contains $F$. Define
  \[
    \HausdorffM^g_\rho(F) \coloneqq
    \inf \Big\{ \sum_i g\paren*[\big]{r(B_i)} :
                 \{B_i\} \text{ is a $\rho$-cover of } F \Big\}.
  \]
  The \emph{(generalised) Hausdorff measure} $\Haus{F}$ of $F$ with respect to
  the dimension function $g$ is defined as
  \[
    \Haus{F}
    \coloneqq \lim_{\rho \to 0} \HausdorffM^g_\rho(F)
    =         \sup_{\rho   > 0} \HausdorffM^g_\rho(F).
  \]
  See \cite[Chapter~4]{Matti1995} for more details.
\end{definition}

\begin{example}
  \label{ex:hausdorff_measure}
  Given $s > 0$, the usual Hausdorff $s$-measure $\HausdorffM^s$ coincides with
  $\HausdorffM^g$ where $g(r) = r^s$. In particular, when $s$ is an integer
  $\HausdorffM^s$ is a constant multiple of the $s$-dimensional Lebesgue
  measure.
\end{example}

Define $\dualappr_{n,m+1}^*(\approxfdim,\approxfcodim; d)$ to be the set of
$\bm{x} \in \R^{m+1}$ such that
\[
  \max_{0 \leq k < d} \abs{x_k - \alpha_k}
    < \pointapproxf{\polyheight{\bm{\alpha}}}
  \quad \text{and} \quad
  \max_{d \leq k \leq m} \abs{x_k - \alpha_k}
    < \pointapproxf[\approxfcodim]{\polyheight{\bm{\alpha}}}
\]
for infinitely many $\bm{\alpha} \in \Alg{m+1}{n}$.
When $\bm{f}$ parametrises a $d$-dimensional manifold $\manifold$ and
$f_k(\bm{x}) = x_k$ for every $0 \leq k < d$, will also write
$\dualappr_{n,\bm{f}}^*(\approxfdim,\approxfcodim)$ for the set of
$\bm{x} \in \domf$ such that
$\bm{f}(\bm{x}) \in \dualappr_{n,m+1}^*(\approxfdim,\approxfcodim; d)$.
This section is devoted to the proof of the following Proposition, of which
Theorem~\ref{thm:polynomial_jarnik_divergence} is a direct consequence.

\begin{proposition}
  \label{prop:manifold_khincin_point_divergence}
  Let $\approxfdim,\approxfcodim \colon \R^+ \to \R^+$ be decreasing functions
  which satisfy \eqref{eq:cnd:approxfcodim_lower_bound}, and let $g$ be a
  dimension function such that $r^{-d}g(r)$ is non-increasing.
  Also assume that $r^{-\gamma}g(r)$ is increasing for some $\gamma > 0$, and
  that it satisfies \eqref{eq:cnd:ubiquity_divergence_dimension_function}.
  Further suppose that $\bm{f}$ is Lipschitz continuous,
  that $\Vand(\bm{f}) \neq 0$, and
  that $\bm{f}$ satisfies condition~\eqref{eq:cnd:minors_d} on $\domf$.
  Then
  \[
    \Haus{\dualappr_{n,\bm{f}}^*(\psi)} = \Haus{\domf}
    \quad \text{if} \quad
    \sum_{Q = 1}^{\infty}
      Q^{n-m-1+d} \approxfcodim(Q)^{m+1-d}
      g\paren{\pointapproxf{Q}}
    = \infty.
  \]
\end{proposition}

\begin{remark}
  \label{rmk:approximation_set_inclusion}
  There is a constant $\capprmax > 0$, depending only on $n$ and $\manifold$,
  such that
  \[
    \dualappr_{n,\bm{f}}^*(\approxfdim,\approxfcodim)
    \subseteq \dualappr_{n,\bm{f}}(\capprmax\approxfdim, \capprmax\approxfcodim).
  \]
  Indeed, suppose that $\bm{y} \in \dualappr_{n,\bm{f}}^*(\psi)$, and
  let $\bm{\alpha} \in \Alg{m+1}{n}$ be such that
  $\norm{\bm{y} - \bm{\alpha}} < \pointapproxf{\polyheight{\bm{\alpha}}}$. If
  $P$ is the minimum polynomial of $\alpha_0$ (and hence of $\alpha_k$ for every
  $0 \leq k \leq m+1$), then by the Mean Value Theorem we have
  \[
    \begin{aligned}
      \abs{P(y_k)} &= \abs{P(y_k) - P(\alpha_k)} \\
                   &\leq \abs{y_k - \alpha_k}
                         \sup_{\bm{z} \in \manifold} \abs{P'(z_k)} \\
                   &< \capprmax \pointapproxf[\psi_k]{\polyheight{\bm{\alpha}}}
                      \polyheight{P} \\
                   &= \capprmax \psi_k(\polyheight{\bm{\alpha}})
    \end{aligned}
    \qquad \text{with} \qquad
    \psi_k =
    \begin{cases}
      \psi & \text{for } 0 \leq k < d \\
      \Psi & \text{for } d \leq k \leq m
    \end{cases},
  \]
  since $\domf$ bounded implies that $P'$ is bounded above on $\manifold$, and
  of course $\polyheight{P} = \polyheight{\bm{\alpha}}$. Thus it follows that
  $\bm{y} \in \dualappr_{n,\bm{f}}(\capprmax\approxfdim,\capprmax\approxfcodim)$,
  as required.
\end{remark}

Therefore the convergence part of
Theorem~\ref{thm:independent_polynomial_khinchin} immediately gives the
following partial counterpart of
Proposition~\ref{prop:manifold_khincin_point_divergence}. Here, as before,
$\abs{U}$ denotes the Lebesgue measure of a measurable set $U \subset \R^{m+1}$.

\begin{lemma}
  \label{lm:manifold_khinchin_point_convergence}
  For any function $\psi \colon \R^+ \to \R^+$ we have
  \[
    \Leb{\dualappr_{n,\bm{f}}^*(\approxfdim,\approxfcodim)} = 0
    \quad \text{if} \quad
    \sum_{Q = 1}^{\infty}
      Q^{n-m-1} \approxfdim(Q)^d \approxfcodim(Q)^{m+1-d} < \infty.
  \]
\end{lemma}

Our proof relies on a powerful tool of Diophantine Approximation,
\emph{ubiquitous systems}, adapted to the case of approximation of dependent
quantities like in \cite{Beres2012}. Consider the following setting:
\begin{itemize}
  \item $\Omega$, a compact subset of $\R^d$;
  \item $J$, a countable set;
  \item $\resonant = (R_{\bm{\alpha}})_{\bm{\alpha} \in J}$ a family of points
    in $\Omega$ indexed by $J$, referred to as \emph{resonant points};
  \item a function
    $\weight \colon J \to \R^+, \bm{\alpha} \mapsto \beta_{\bm{\alpha}}$, which
    assigns a \emph{weight} to each $R_{\bm{\alpha}}$ in $\resonant$;
  \item a function $\rho \colon \R^+ \to \R^+$ such that
    $\lim_{r \to \infty} \rho(r) = 0$, referred to as a \emph{ubiquitous
    function}; and
  \item $J(t) = J_{\kappa}(t) \coloneqq
    \{ \bm{\alpha} \in J \colon \weight_{\bm{\alpha}} \leq \kappa^t \}$, assumed
    to be finite for every $t \in \N$, where $\kappa > 1$ is fixed.
\end{itemize}

Furthermore, $B(\bm{x}, r)$ will denote a ball in $\Omega$ with respect
to the sup norm, and for a given function $\approxfubq \colon \R^+ \to \R^+$ we
will also consider the limsup set
\[
  \lsset(\approxfubq)
  \coloneqq \{ \bm{x} \in \Omega :
    \norm{\bm{x} - R_{\bm{\alpha}}} < \approxfubq(\weight_{\bm{\alpha}})
    \text{ for infinitely many } \bm{\alpha} \in J \}.
\]

\begin{definition}
  The pair $(\resonant, \weight)$ is a \emph{locally ubiquitous system in
  $\Omega$ with respect to $\rho$} if for any ball $B \in \Omega$
  \[
    \Leb[\Big]{\bigcup_{\bm{\alpha} \in J(t)} B(\bm{\alpha}, \rho(\kappa^t))
               \cap B}
    \gg \Leb{B}
  \]
  for every $t$ large enough, where the implied constant is absolute.
\end{definition}

Like with \cite[Theorem~1]{BeresV2009}, the following statement can be readily
obtained by combining Corollaries 2 and 3 from \cite{BeresDV2006}:

\begin{theorem}
  \label{thm:ubiquity_divergence}
  In the above setting, suppose that $(\resonant, \weight)$ is a locally
  ubiquitous system in $\Omega$ with respect to $\rho$, and let $g$ be a
  dimension function such that $r^{-d}g(r)$ is non-increasing. Furthermore,
  supose that $r^{-\gamma}g(r)$ is increasing for some $\gamma > 0$, and that
  there are constants $r_0,\cubdiva,\cubdivb \in (0,1)$ such that
  \begin{equation*}
    r^{\gamma} g(\cubdiva r) \leq \cubdivb g(r) (\cubdiva r)^{\gamma}
    \text{ for any } r \in (0,r_0).
  \end{equation*}
  Also assume that $\approxfubq$ is decreasing and that
  \begin{equation}
    \label{eq:cnd:ubiquity_divergence_regular}
    \limsup_{t \to \infty}
      \frac{\approxfubq(\kappa^{t+1})}{\approxfubq(\kappa^t)} < 1.
  \end{equation}
  Then
  \[
    \Haus[]*[\big]{\lsset(\approxfubq)} = \Haus{\Omega}
    \quad \text{if} \quad
    \sum_{t = 0}^\infty \frac{g(\approxfubq(\kappa^t))}{\rho(\kappa^t)^d} = \infty.
  \]
\end{theorem}

Now let $\bm{f}$, $\approxfdim$ and $\approxfcodim$ be as in
Proposition~\ref{prop:manifold_khincin_point_divergence}. Since $\domf$ is
assumed to be bounded, for every integer
$q \geq 2$ we may find a compact subset $\domf_q \subset \domf$ such that
$\Leb{\domf_q} \geq (1 - \frac{1}{q}) \Leb{\domf}$. It follows that
$\Haus{\domf} = \lim_{q \to \infty} \Haus{\domf_q}$, so it suffices to prove the
proposition with $\domf_q$ in place of $\domf$ for any fixed $q$.
For ease of notation, given $\bm{y} = (y_0,\dotsc,y_m) \in \R^{m+1}$ we will
write $\hat{\bm{y}}$ for $(y_0,\dotsc,y_{d-1})$. Then let
$\Omega \coloneqq \domf_q$ and define
\begin{gather*}
  J \coloneqq \left\{ \bm{\alpha} \in \Alg{m+1}{n} :
    \hat{\bm{\alpha}} \in \Omega
    \text{ and }
    \max_{d \leq k \leq m} \abs{f_k(\hat{\bm{\alpha}}) - \alpha_k}
    < \frac{1}{2} \pointapproxf[\approxfcodim]{\polyheight{\alpha}} \right\} \\
  \resonant \coloneqq (\hat{\bm{\alpha}})_{\bm{\alpha} \in J}
  \qquad
  \beta_{\bm{\alpha}} \coloneqq \polyheight{\bm{\alpha}}.
\end{gather*}
Also let
\[
  \rho(Q)
  = \rho_0 \paren*[\big]{Q^{n-m+d} \approxfcodim(Q)^{m+1-d}}^{-\frac{1}{d}}
\]
for some constant $\rho_0 > 0$ to be determined later, and observe that
\eqref{eq:cnd:approxfcodim_lower_bound} implies
\begin{align*}
  \rho(Q)
  &\ll \paren*[\big]{Q^{n-m+d} Q^{\frac{m-n}{m+1}(m+1-d)}}^{-\frac{1}{d}} \\
  &= \paren*[\big]{Q^{n-m+d} Q^{m - n - \frac{d}{m+1}}}^{-\frac{1}{d}} \\
  &= Q^{-1 + \frac{1}{m+1}},
\end{align*}
which shows that $\rho(Q) \to 0$ as $Q \to \infty$.

\begin{lemma}
  \label{lm:lipschitz_bound}
  Suppose that $f_d,\dotsc,f_m$ are Lipschitz continuous with constant bounded
  above by $\lipf$. If $\bm{y} \in \R^{m+1}$ is such that
  $\hat{\bm{y}} \in \domf$,
  \[
    \max_{0 \leq k < d} \abs{x_k - y_k} < \Theta_{\bm{x}}
    \quad \text{and} \quad
    \max_{d \leq k \leq m} \abs{f_k(\bm{x}) - y_k} < \Theta_{\bm{f}}
  \]
  for some $\Theta_{\bm{x}}, \Theta_{\bm{f}} > 0$, then
  \[
    \max_{d \leq k \leq m} \abs{f_k(\hat{\bm{y}}) - y_k}
    < \Theta_{\bm{f}}
      \paren{1 + \lipf \frac{\Theta_{\bm{x}}}{\Theta_{\bm{f}}}}.
  \]
  \begin{proof}
    Simply observe that, by the triangle inequality,
    \begin{align*}
      \abs{f_k(\hat{\bm{y}}) - y_k}
      &\leq \abs{f_k(\hat{\bm{y}}) - f_k(\bm{x})} + \abs{f_k(\bm{x}) - y_k} \\
      &<     \lipf \norm{\bm{x} - \hat{\bm{y}}} + \abs{f_k(\bm{x}) - y_k} \\
      &<     \lipf \Theta_{\bm{x}} + \Theta_{\bm{f}} \\
      &=     \Theta_{\bm{f}}
             \paren{1 + \lipf \frac{\Theta_{\bm{x}}}{\Theta_{\bm{f}}}}.
      \qedhere
    \end{align*}
  \end{proof}
\end{lemma}

\begin{lemma}
  \label{lm:ubiquitous_system}
  Let $J,\resonant,\weight,\rho$ be as above, and suppose that $\bm{f}$ is
  Lipschitz continuous. Then there is a choice of $\rho_0 > 0$ such that
  $(\resonant,\weight)$ is a locally ubiquitous system in $\Omega$ with respect
  to $\rho$.
  \begin{proof}
    Fix a ball $B \subset \Omega$ and let
    \[
      \begin{cases}
        \psi_k(Q) = Q\rho(Q) & \text{for } 0 \leq k < d \\
        \psi_k(Q) = \approxfcodim(Q) & \text{for } d \leq k \leq m \\
        \varphi_k(Q) = Q & \text{for } m < k \leq n.
      \end{cases}
    \]
    Then observe that $\frac{\psi_k(Q)}{Q}$ is decreasing for every
    $0 \leq k \leq m$, and that
    \[
      \bm{\psi}(Q)\bm{\varphi}(Q)
      = \rho(Q)^d \approxfcodim(Q)^{m+1-d} Q^{n-m+d}
      = \rho_0^d.
    \]
    Furthermore, by condition \eqref{eq:cnd:approxfcodim_lower_bound} we
    know that $\approxfcodim(Q) \gg Q^{\frac{m-n}{m+1}}$, which implies
    $\approxfcodim(Q) \gg Q\rho(Q)$. Therefore for every $1 \leq \tau \leq m+1$
    and every choice of $I \in \indexsets{m}{\tau}$ we have that
    \[
      \psi_I(Q) \ll \approxfcodim(Q)^{\tau} \ll 1
    \]
    for every $Q$ large enough, thus
    we may apply Corollary~\ref{cor:n+1_points} with
    $\nu(\cdot) = \Leb{}$ and $B$ in place of $\domf$. Hence for any fixed
    $0 < \theta < 1$ we find a set $B_\theta \subseteq B$ with
    $\Leb{B_\theta} > \theta \Leb{B}$ and a constant $c > 0$
    such that, for $t$ large enough, every $\bm{x} \in B_\theta$ admits $n+1$
    points $\bm{\alpha} \in \alg{m+1}{n}{c\kappa^t}$ with
    \[
      \max_{0 \leq k < d} \abs{x_k - \alpha_k} < c \rho(\kappa^t)
      \quad \text{and} \quad
      \max_{d \leq k \leq m} \abs{f_k(\bm{x}) - \alpha_k}
        < c \pointapproxf[\approxfcodim]{\kappa^t}.
    \]
    Now, again because of $\approxfcodim(Q) \gg Q\rho(Q)$, by
    Lemma~\ref{lm:lipschitz_bound} it follows that
    \[
      \max_{d \leq k \leq m} \abs{f_k(\hat{\bm{\alpha}}) - \alpha_k}
        < c \hat{c} \pointapproxf[\approxfcodim]{\kappa^t}
    \]
    for some $\hat{c} > 1$.
    Finally, observe that Remark~\ref{rmk:n+1_points_constant} and
    equation~\eqref{eq:ctailoredy} show that we can choose $c$ by manipulating
    the value of $\rho_0$. In particular, we can ensure that
    $c < \hat{c}^{-1}$, thus $\bm{\alpha} \in J(t)$ and
    \[
      \Leb[\Big]{\bigcup_{\bm{\alpha} \in J(t)} B(\bm{\alpha}, \rho(\kappa^t))
                 \cap B}
      \geq \Leb{B_\theta}
      > \theta \Leb{B}. \qedhere
    \]
  \end{proof}
\end{lemma}

\begin{note}
  Condition \eqref{eq:cnd:t_positive_d} is actually satisfied even in the
  absence of \eqref{eq:cnd:approxfcodim_lower_bound}. Indeed, using the
  fact that $\approxfcodim$ is decreasing and that $\bm{\psi}\bm{\varphi}$ is
  constant, one can show that $\psi_I$ is decreasing for every
  $I = (m+1-\tau,\dotsc,m)$ where $1 \leq \tau \leq m+1$.
\end{note}

\begin{proof}[Proof of Proposition~\ref{prop:manifold_khincin_point_divergence}]
  Note that since $\approxfdim \in O(\approxfcodim)$, there is a
  $\cggapprox > 0$ such that $\approxfcodim(Q) > \cggapprox \approxfdim(Q)$ for
  any integer $Q > 0$.
  Then let $\approxfubq(Q) = \frac{\cggapprox}{2\lipf} \pointapproxf{Q}$, where
  $\lipf$ is as in Lemma~\ref{lm:lipschitz_bound}. It is clear that this choice
  of $\approxfubq$ satisfies \eqref{eq:cnd:ubiquity_divergence_regular}: indeed,
  \[
    \limsup_{t \to \infty}
      \frac{\approxfubq(\kappa^{t+1})}{\approxfubq(\kappa^t)}
    = \frac{1}{\kappa}
      \limsup_{t \to \infty}
        \frac{\approxfdim(\kappa^{t+1})}{\approxfdim(\kappa^t)}
    < \frac{1}{\kappa}
  \]
  since $\approxfdim$ is assumed to be decreasing.
  The proposition will follow as an immediate consequence of
  Theorem~\ref{thm:ubiquity_divergence} once we've shown that
  $\lsset(\approxfubq) \subseteq
   \dualappr_{n,\bm{f}}^*(\approxfdim,\approxfcodim)$.

  If $\bm{x} \in \lsset(\approxfubq)$, then there are infinitely many
  $\bm{\alpha} \in \Alg{m+1}{n}$ such that
  \[
    \max_{0 \leq k < d} \abs{x_k - \alpha_k}
      < \frac{\cggapprox}{2\lipf} \pointapproxf{\polyheight{\bm{\alpha}}}
    \quad \text{and} \quad
    \max_{d \leq k \leq m} \abs{f_k(\hat{\bm{\alpha}}) - \alpha_k}
    < \frac{1}{2} \pointapproxf[\approxfcodim]{\polyheight{\bm{\alpha}}}.
  \]
  Therefore the same argument of Lemma~\ref{lm:lipschitz_bound} gives
  \[
    \max_{d \leq k \leq m} \abs{f_k(\bm{x}) - \alpha_k}
    < \pointapproxf[\approxfcodim]{\polyheight{\bm{\alpha}}}.
  \]
  It follows that
  $\bm{x} \in \dualappr_{n,\bm{f}}^*(\approxfdim,\approxfcodim)$, since we may
  assume without loss of generality that $\lipf \geq \frac{1}{2}$.
  The proof is concluded by observing that by Cauchy's Condensation Test
  \[
    \sum_{t = 0}^{\infty} \frac{g(\approxfubq(\kappa^t))}{\rho(\kappa^t)^d}
    = \rho_0^{-d} \sum_{Q = 1}^{\infty}
      \kappa^{t(n-m+d)} \approxfcodim(\kappa^t)^{m+1-d}
      g\paren{\frac{\cggapprox}{2\lipf} \pointapproxf{\kappa^t}}
    = \infty
  \]
  if and only if
  \[
    S_1 \coloneqq \sum_{Q = 1}^{\infty}
      Q^{n-m-1+d} \approxfcodim(Q)^{m+1-d}
      g\paren{\frac{\cggapprox}{2\lipf} \pointapproxf{Q}}
    = \infty,
  \]
  and that the same argument of
  \cite[Lemma~3.2]{Pezzo2020} shows that the latter happens if and only if
  \[
    S_2 \coloneqq \sum_{Q = 1}^{\infty}
      Q^{n-m-1+d} \approxfcodim(Q)^{m+1-d}
      g\paren{\pointapproxf{Q}}
    = \infty
  \]
  when $g$ is increasing and $\approxfdim$ is decreasing. Indeed, note that
  without loss of generality we may assume that $2\lipf > \cggapprox$, and let
  $c \coloneqq \frac{2\lipf}{\cggapprox} > 1$. Furthermore, for ease of notation
  let
  \[
    \sigma(z, Q) \coloneqq
      Q^{n-m-1+d} \approxfcodim(Q)^{m+1-d}
      g\paren{z \pointapproxf{Q}}.
  \]
  Now, on one hand $g$ increasing immediately implies that $S_1 \leq S_2$.
  On the other hand,
  \begin{align*}
    S_2
    &= \sum_{Q = 1}^{c-1} \sigma(1,Q)
     + \sum_{q = 1}^{\infty}
       \sum_{cq \leq Q < c(q+1)} \sigma(1,Q) \\
    &\ll \sum_{q = 1}^{\infty} \sigma(1.cq) \\
    &\ll \sum_{q = 1}^{\infty} \sigma(c^{-1}.q) \\
    &= S_1,
  \end{align*}
  where the last inequality is due to the fact that $\psi$ is decreasing and $g$
  is increasing.
\end{proof}

\section{Proofs of Theorem~\ref{thm:bound_over_an_interval}
         and Corollary~\ref{cor:metric_convergence}}
\label{sec:proof_of_the_main_theorem}

We will first prove the following local version of
Theorem~\ref{thm:bound_over_an_interval}, which can then be extended via a
compactness argument.

\begin{theorem}
  \label{thm:local_bound_over_an_interval}
  Under the hypotheses of Theorem~\ref{thm:bound_over_an_interval}, fix a point
  $\bm{x} \in \domf \cap \supp\nu$ and let $B \ni \bm{x}$ be a ball such that
  $\tilde{B} = 3^{n+1}B \subset \domf$.
  Then for $Q$ large enough we may find constants
  $C,\rho > 0$, the latter dependant on $B$, such that
  \[
    \nu\paren{\maxderset^n_{\bm{f}}(Q, B)}
    \leq C \paren{ \frac{\bm{\psi}(Q) \bm{\varphi}(Q)}
                        {\rho^{n+1}} }^\frac{\alpha}{n+1}
         \nu(B).
  \]
\end{theorem}

\begin{proof}[Proof of Theorem~\ref{thm:bound_over_an_interval} given
              Theorem~\ref{thm:local_bound_over_an_interval}]
   Let $\domfgood{\theta} \subset \domf$ be a compact subset such that
   $\nu(\domfgood{\theta}) \geq \nu(\domf)$, which exists because
   $\domf$ is bounded and $\nu$ is Radon. Then note that, since
   $\domf \cap \supp\nu$ is contained in the interior of $\domf$ by hypothesis,
   for every $\bm{x} \in \domfgood{\theta}$ we may find a ball
   $B_{\bm{x}} \ni \bm{x}$ as in Theorem~\ref{thm:local_bound_over_an_interval},
   as well as the respective constants $C_{\bm{x}}$ and $\rho_{\bm{x}}$.
   Hence by compactness there is a finite subset
   $\{x_k\}_{k \in K} \subset \domfgood{\theta}$ such that
   $\{B_{\bm{x}_k}\}_{k \in K}$ is an open cover of
   $\domfgood{\theta}$. Therefore the result follows by observing that
   \[
      \nu\paren{\maxderset^n_{\bm{f}}(Q, \domfgood{\theta})}
      \leq \sum_{k \in K} \nu\paren{\maxderset^n_{\bm{f}}(Q, B_{\bm{x}_k})}
   \]
   and by taking $C = \max_K C_{\bm{x}_k}$ and $\rho = \max_K \rho_{\bm{x}_k}$.
\end{proof}

Through the Dani-Kleinbock-Margulis correspondence between Diophantine
Approximation and flows on homogeneous spaces \cite{Dani1986,KleinM1998}, we
will reinterpret the problem of finding points $x \in \domf$ for which
\eqref{eq:d_dimensional_inequalities} has a solution as a shortest vector
problem. First, we expand \eqref{eq:d_dimensional_inequalities} into the $m+1$
systems of inequalities
\begin{equation}
  \label{eq:split_d_dimensional_inequalities}
  \begin{aligned}
    &\begin{cases}
      \abs{P(f_k(\bm{x}))} < \psi_k(Q) & \text{for } 0 \leq k \leq m \\
      P'(f_h(\bm{x})) \leq \varphi_{m+1}(Q) \\
      \abs{a_k} \leq \varphi_k(Q) & \text{for } m+1 < k \leq n
    \end{cases}
    &\qquad 0 &\leq h \leq m,
  \end{aligned}
\end{equation}
and observe that these can be rewritten in matrix form using the matrices
$U_{\bm{f}}^h$.

However, to be able to view this as a smallest vector problem we also need to
rescale the inequalities.
Consider the scaling matrix
\begin{equation}
  \label{eq:definition_of_g}
  g_{\bm{t}}
  \coloneqq \diag\paren{e^{t_0}, \dotsc, e^{t_m},
                        e^{-t_{m+1}}, \dotsc, e^{-t_n}
                       },
\end{equation}
where $\bm{t} = (t_0,\dotsc,t_n) \in \R^{n+1}$ is such that
\[
  t_{\intrange[0]{m}} = t_{\intrange[m+1]{n}},
\]
and where for every $1 \leq \tau \leq n+1$ and $I \in \indexsets{n}{\tau}$ we
defined
\begin{equation}
  \label{eq:notation_t_range}
  t_I = \sum_{i \in I} t_i.
\end{equation}
Then we need $\delta = \delta(Q) > 0$ such that
\begin{equation}
  \label{eq:definition_of_delta}
  \begin{cases}
    \delta = e^{t_k} \psi_k(Q) & \text{for } 0 \leq k \leq m \\
    \delta = e^{-t_k} \varphi_k(Q) & \text{for } m < k \leq n,
  \end{cases}
\end{equation}
and multiplying those $n+1$ equations together we see that
\begin{equation}
  \label{eq:delta_as_product}
  \delta^{n+1} = \bm{\psi}(Q) \bm{\varphi}(Q).
\end{equation}
Therefore, taking logarithms we may rewrite $t_k$ in terms of $\psi_k$ and
$\varphi_k$, as
\begin{equation}
  \label{eq:parameters_from_approximation_functions}
  \begin{cases}
    t_k = \log\delta - \log\psi_k(Q) & \text{for } 0 \leq k \leq m \\
    t_k = \log\varphi_k(Q) - \log\delta & \text{for } m < k \leq n.
  \end{cases}
\end{equation}
We can now see that \eqref{eq:d_dimensional_inequalities} has a solution
for a given $\bm{x} \in \domf$ if and only if for every $0 \leq h \leq m$ the
lattice $g_{\bm{t}}U_{\bm{f}}^h(\bm{x}) \Z^{n+1}$ has a non-zero vector with
sup-norm at most $\delta$, thus we have indeed reduced to a shortest vector
problem. In other words,
\[
  \maxderset^n_{\bm{f}}(Q, B)
  = \left\{\bm{x} \in B :
           \lambda\paren{g_{\bm{t}}U_{\bm{f}}^h(\bm{x}) \Z^{n+1}}
           < \paren*[\big]{\bm{\psi}(Q) \bm{\varphi}(Q)}^{\frac{1}{n+1}}
    \right\},
\]
where $\lambda(\Gamma) = \inf_{v \in \Gamma \setminus \{0\}} \norm{v}$ denotes
the length of the shortest vector in a discrete subgroup
$\Gamma \subset \R^{n+1}$.

\begin{remark}
  \label{rmk:lower_bound_t}
  Condition~\eqref{eq:cnd:t_positive_d} of
  Theorem~\ref{thm:bound_over_an_interval} is equivalent to asking that for
  every $1 \leq \tau \leq m+1$ there is a choice of $I \in \indexsets{m}{\tau}$
  such that $t_I = t_I(Q)$ is bounded below. Indeed,
  by~\eqref{eq:parameters_from_approximation_functions}
  \begin{align*}
    t_I &= \tau \log \delta - \sum_{i \in I} \log \psi_i(Q) \\
        &= \frac{\tau}{n+1} \log\paren{\bm{\psi}(Q)\bm{\varphi}(Q)}
           - \log \psi_I(Q) \\
        &\geq c
  \end{align*}
  precisely when $\psi(Q)\varphi(Q) \geq e^c \psi_I(Q)^{\frac{n+1}{\tau}}$.
\end{remark}

\begin{example}
  \label{ex:original_case}
  Let $d = m = 1$, as in the context of
  Theorem~\ref{thm:bernik_goetze_gusakova}. Furthermore, let
  $\psi_0(Q) = \psi_1(Q) = Q^{-\frac{n-1}{2}}$,
  $\varphi_2(Q) = \varepsilon^{n+1} Q$
  and $\varphi_3(Q) = \dotsb = \varphi_n(Q) = Q$. Then by
  \eqref{eq:delta_as_product} we have $\delta = \varepsilon$.
  Moreover, the
  equations~\eqref{eq:parameters_from_approximation_functions} become
  \begin{equation}
    \label{eq:parameters_original_case}
    \begin{cases}
      t_k = \log \varepsilon + \frac{n-1}{2} \log Q
      & \text{for } 0 \leq k \leq 1 \\
      t_2 = n\log \varepsilon + \log Q \\
      t_k = \log Q - \log \varepsilon
      & \text{for } 3 \leq k \leq n.
    \end{cases}
  \end{equation}
  Therefore $t_{\intrange{1}} = 2 \log \varepsilon + (n-1) \log Q$, which in
  particular gives that $t_{\intrange{1}} \geq c$ for
  \[
    \log\varepsilon \geq \frac{\log c}{2} - \frac{n - 1}{2} \log Q,
  \]
  i.e. $\varphi_2(Q) \gg Q^{\frac{3-n^2}{2}}$.
  Then $\bm{\psi}(Q)\bm{\varphi}(Q) \gg Q^{\frac{1-n^2}{2}}$ and we easily see
  that condition~\eqref{eq:cnd:t_positive_d} of
  Theorem~\ref{thm:bound_over_an_interval} is satisfied for all choices of $I$,
  since
  \[
    \bm{\psi}(Q)^{\frac{n+1}{2}} = \psi_0(Q)^{n+1} = \psi_1(Q)^{n+1}
    = Q^{\frac{1-n^2}{2}}.
  \]
\end{example}

The main tool in our proof will be the following Theorem from \cite{Klein2008}.
Here $\primtuples_\tau$ denotes the set of elements
\[
  \bm{w} = \bm{w}_1 \wedge \dotsb \wedge \bm{w}_\tau \in \bigwedge^\tau \Z^{n+1}
\]
where $\{\bm{w}_1,\dotsc,\bm{w}_\tau\}$ is a \emph{primitive $\tau$-tuple}, i.e.
it can be completed to a basis of $\Z^{n+1}$. Furthermore, $\norm{\cdot}$ will
denote both the sup-norm and the norm it induces on $\bigwedge \R^{n+1}$.

\begin{note}
  The elements of $\primtuples_\tau$ can be identified with the primitive
  subgroups of $\Z^{n+1}$ of rank $\tau$, i.e. those non-zero subgroups
  $\Gamma \subseteq \Z^{n+1}$ of rank $\tau$ such that
  $\Gamma = \Gamma_\R \cap \Z^{n+1}$, where $\Gamma_\R$ denotes the linear
  subspace generated by $\Gamma$ in $\R^{n+1}$.
  Therefore, up to a sign they can also be identified with the rational
  $\tau$-dimensional subspaces of $\R^{n+1}$.
\end{note}

\begin{theorem}[{\cite[Theorem~2.2]{Klein2008}}]
  \label{thm:km_sublattices}
  Fix $n,N \in \N$ and $\tilde{C}, D, \alpha, \rho > 0$. Given an
  $N$-Besicovitch metric space $X$, let $B$ be a ball in $X$ and $\nu$ be a
  measure which is $D$-Federer on $\tilde{B} = 3^{n+1}B$.
  Suppose that $\eta \colon \tilde{B} \to \GL_{n+1}(\R)$ is a map such that for
  every $1 \leq \tau \leq n+1$ and for every $\bm{w} \in \primtuples_\tau$:
  \begin{enumerate}
    \item \label{lst:cnd:km_sublattices_goodness}
      the function $\bm{x} \mapsto \norm{\eta(\bm{x})\bm{w}}$ is
      $(\tilde{C},\alpha)$-good on $\tilde{B}$ with respect to $\nu$, and
    \item \label{lst:cnd:km_sublattices_rho}
      $\norm{\eta(\cdot)\bm{w}}_{\nu,B} \geq \rho^\tau$.
  \end{enumerate}
  Then for any $0 < \delta \leq \rho$ we have
  \[
    \nu \left( \left\{ \bm{x} \in B :
                       \lambda\left( \eta(\bm{x})\Z^{n+1} \right) < \delta
               \right\}
        \right)
    \leq C \left( \frac{\delta}{\rho} \right)^\alpha \nu(B)
  \]
  with $C = (n+1)\tilde{C}(ND^2)^{n+1}$.
\end{theorem}

\begin{note}
  In light of Lemma~\ref{lm:good_function_properties}, we may extend this to
  $\delta > \rho$ as well, since we may always exchange $\tilde{C}$ with
  $\max\{\tilde{C}, (n+1)^{-1}(ND^2)^{-n-1}\}$, so that $C \geq 1$.
\end{note}

For our purposes we would like to take
$\eta(\bm{x}) = g_{\bm{t}} U^h_{\bm{f}}(\bm{x})$, and to show that it satisfies
hypotheses \ref{lst:cnd:km_sublattices_goodness} and
\ref{lst:cnd:km_sublattices_rho} we will need the following Lemma.
Here, for each $I \subseteq \indexsets{n}{\tau}$ we will denote
by $\bm{e}_I$ the standard basis element
$\bm{e}_{i_1} \wedge \dotsm \wedge \bm{e}_{i_\tau}$ of
$\bigwedge^\tau \R^{n+1}$.

\begin{lemma}
  \label{lm:exterior_product_components}
  Let $\bm{w} = \bm{w}_1 \wedge \dotsm \wedge \bm{w}_\tau \in \primtuples_\tau$
  and let $A$ be an $(n+1) \times (n+1)$ matrix. Then, for every
  $I \subseteq \indexsets{n}{\tau}$, the component of $A\bm{w}$
  corresponding to $\bm{e}_I$ is an integer linear combination of the minors
  $\minor{A}{I,J}$, where $J$ runs through $\indexsets{n}{\tau}$. Furthermore,
  the coefficients are inrependent from $I$ and not all zero.
  \begin{proof}
    Let $W = (\bm{w}_1 | \dotsb | \bm{w}_\tau)$ be the matrix obtained by
    juxtaposition of the vectors $\bm{w}_1,\dotsc,\bm{w}_\tau$, and recall the
    well-known fact that the $\bm{e}_I$ component of $\bm{w}$ is just the
    $\tau \times \tau$ minor $\minor{W}{I,\intrange*{\tau}}$
    of $W$ (see e.g.  \cite[Chapter~10, Section~3]{ShafaR2013}), where
    $\intrange*{\tau} = \{1,\dotsc,\tau\}$. Now observe that
    \[
      A\bm{w} = (A\bm{w}_1) \wedge \dotsm \wedge (A\bm{w}_\tau)
      = AW (\bm{e}_1 \wedge \dotsm \wedge \bm{e}_\tau)
    \]
    and that $(A\bm{w}_1 | \dotsb | A\bm{w}_\tau) = AW$. Finally, the statement
    follows by the Cauchy-Binet formula (see e.g.
    \cite[Example~10.31]{ShafaR2013} or
    \cite[Cauchy-Binet~Corollary, p.~214]{BroidW1989}), i.e.
    \[
      \minor{AW}{I,\intrange*{\tau}}
      = \sum_{J \in \indexsets{n}{\tau}} \minor{A}{I,J} \minor{W}{J,\intrange*{\tau}}.
      \qedhere
    \]
  \end{proof}
\end{lemma}

It follows that the component of $g_{\bm{t}} U^h_{\bm{f}}(\bm{x})\bm{w}$
corresponding to $\bm{e}_I$ is of the form
\[
  e^{t_I} \sum_{J \in \indexsets{n}{\tau}} c_J \minor{U^h_{\bm{f}}}{I,J}
\]
with $c_J \in \Z$ not all zero and independent from $I$. Therefore we have that
\[
  \norm{g_{\bm{t}} U^h_{\bm{f}}(\bm{x})\bm{w}}
  \gg \abs*[\Bigg]{\sum_{J \in \indexsets{n}{\tau}} c_J \minor{U^h_{\bm{f}}}{I,J}}
\]
as long as $e^{t_I}$ is bounded below, which we know from
Remark~\ref{rmk:lower_bound_t} to be guaranteed by
condition~\eqref{eq:cnd:t_positive_d} of
Theorem~\ref{thm:bound_over_an_interval}.
We can now prove that for every $1 \leq \tau \leq n+1$ the norm of
$g_{\bm{t}} U^h_{\bm{f}}(\bm{x})\bm{w}$ is bounded below uniformly in
$w \in \primtuples_\tau$.

This is straighforward for $\tau = n+1$, since in that case
\[
  \norm*[\big]{g_{\bm{t}} U^h_{\bm{f}}(\bm{x})\bm{w}}
  = \abs{c_{\intrange{n}} \det U^h_{\bm{f}}}
  \geq \abs{\det U^h_{\bm{f}}}
  > 0
\]
by remark~\ref{rmk:non_zero_determinant}. When $\tau \leq n$, on the other hand,
condition~\eqref{eq:cnd:derivative_bound} of
Theorem~\ref{thm:bound_over_an_interval} guarantees that
\[
  t_{m+1} = \sum_{i = 0}^m t_i - \sum_{i = m+2}^n t_i
\]
is bounded below, hence we may always find an index set
$I \in \indexsets{n}{\tau}$ such that $m+1 \notin I$ and $e^{t_I}$ is bounded
below. But then
\[
  \grass[I][U_{\bm{f}}^h]
  = \grass[I][M_{\bm{f}}]
  = \grass[\tilde{I}][M_{\bm{f}}]
\]
for some $\tilde{I} \in \indexsets{m}{\tilde{\tau}}$ and
$1 \leq \tilde{\tau} \leq m+1$. Therefore it is enough to check that for every
such $\tilde{I}$ and $\tilde{\tau}$
\begin{equation}
  \label{eq:grass_bounded_below}
  \norm{\bm{c} \cdot \grass[I][M_{\bm{f}(\bm{x})}]}_{\nu,B}
  = \sup_{\bm{x} \in B \cap \supp\nu}
    \abs{\bm{c} \cdot \grass[I][M_{\bm{f}(\bm{x})}]}
  \gg 1
\end{equation}
uniformly in non-zero integer vectors $\bm{c}$. By
Lemma~\ref{lm:rational_linear_combination}, this can be guaranteed by requiring
that $\bm{f}_I$ is non-symmetric of degree $n + 1 - \tau$, because when
$\Vand(\bm{f})$ is bounded we have
\begin{equation}
  \label{eq:grass_bounded_by_schur}
  \abs{\bm{c} \cdot \grass[I][M_{\bm{f}(\bm{x})}]}
  \gg \abs{\bm{c} \cdot \Schur[n,\tau][\bm{f}_I(\bm{x})]}
\end{equation}
and from Proposition~\ref{prop:schur_linear_basis} we know that the components
of $\Schur[n,\tau][\bm{T}] = \SchurD{\tau}{n+1-\tau}{\bm{T}}$ form a basis for
the module of symmetric polynomials in $\tau$ variables and degree bounded by
$n + 1 - \tau$.

\begin{lemma}
  \label{lm:kernel_linearly_independent_over_Q}
  Given a continuous map $\bm{g} = (g_0,\dotsc,g_r) \colon \domf \to \R^r$,
  let $\tilde{\bm{g}} = (\tilde{g}_0,\dotsc,\tilde{g}_{\tilde{r}})$ be a basis
  for the linear span $\linspan[\R]{g_0,\dotsc,g_r}$ and let $R$ be the real
  matrix such that $\tilde{\bm{g}} R = \bm{g}$. Then
  $\ker(R) \cap \Q^{r+1} = \{\bm{0}\}$ if and only if $g_0,\dotsc,g_r$ are
  linearly independent over $\Q$.
  \begin{proof}
    The components of $\bm{g}$ are linearly dependent over $\Q$ if and only if
    there is a non-zero $\bm{q} \in \Q^{r+1}$ such that
    \[
      0 = \bm{g} \cdot \bm{q} = \tilde{\bm{g}} R \bm{q},
    \]
    but then it must be that $\bm{q} \in \ker(R)$, since by hypothesis the
    components of $\tilde{\bm{g}}$ are linearly independent over $\R$.
  \end{proof}
\end{lemma}

\begin{lemma}
  \label{lm:rational_linear_combination}
  Let $\bm{g} = (g_1,\dotsc,g_r) \colon \domf \to \R^r$ be a continuous map with
  components linearly independent over $\Q$. Then there is a $\rho > 0$ such
  that $\norm{g}_{\nu,B} \geq \rho$ for every integer linear combination $g$ of
  the components of $\bm{g}$.
  \begin{proof}
    Let $\tilde{\bm{g}}$ be a basis for the linear span
    $\linspan[\R]{g_0,\dotsc,g_r}$. Further, let $\sphere[r]$ be the unit sphere
    in $\R^{r+1}$ and note that if $\bm{b} \in \Z^r \setminus \{0\}$, then
    $\tilde{\bm{b}} \coloneqq \frac{\bm{b}}{\norm{\bm{b}}}
     \in \sphere[r] \cap \Q^{r+1}$.
    Therefore
    \[
      \min_{\bm{b} \in \Z^r \setminus \{0\}}
        \norm{\bm{g} \cdot \bm{b}}_{\nu,B}
      \geq \min_{\tilde{\bm{b}} \in \sphere[r]}
        \norm*{\bm{g} \cdot \tilde{\bm{b}}}_{\nu,B}
      = \min_{\tilde{\bm{b}} \in \sphere[r]}
        \norm*{\tilde{\bm{g}} \cdot (R\tilde{\bm{b}})}_{\nu,B}
      \eqqcolon \rho,
    \]
    which is well defined since $\tilde{\bm{g}} R \tilde{\bm{b}}$ is continuous
    in $\tilde{\bm{b}}$ and $\sphere[r]$ is compact. Finally,
    Lemma~\ref{lm:kernel_linearly_independent_over_Q} implies that $\rho > 0$.
  \end{proof}
\end{lemma}

Having shown that $\eta(\bm{x}) = g_{\bm{t}} U^h_{\bm{f}}(\bm{x})$ satisfies
condition~\ref{lst:cnd:km_sublattices_rho} of
Theorem~\ref{thm:km_sublattices}, we note that \eqref{eq:grass_bounded_by_schur}
implies that $\eta$ satisfies condition~\ref{lst:cnd:km_sublattices_goodness} as
well.
Indeed, write $\varpi$ for
$\norm*[\big]{g_{\bm{t}} U^h_{\bm{f}}(\bm{x})\bm{w}}_{\nu,B}$ and $\varpi_I$ for
the component of $g_{\bm{t}} U^h_{\bm{f}}(\bm{x})\bm{w}$ corresponding to
$\bm{e}_I$. Since $B$ is bounded and $\varpi,\varpi_I$ are continuous,
\eqref{eq:grass_bounded_below} implies that $\varpi \leq c \varpi_I$ on
$B \cap \supp\nu$ for some $c > 0$. Furthermore,
\eqref{eq:grass_bounded_by_schur} shows that $\varpi_I$ is $(C, \alpha)$-good on
$B$ with respect to $\nu$, since $(\Schur[][\bm{f}_I], \nu)$ is
$(C, \alpha)$-good by hypothesis. Therefore by
Lemma~\ref{lm:good_function_properties:ratio} we have that $\varpi$ is
$(c^\alpha C, \alpha)$-good on $B$ with respect to $\nu$.

This concludes the proof of Theorem~\ref{thm:local_bound_over_an_interval}.

\begin{proof}[Proof of Corollary~\ref{cor:metric_convergence}]
  Note that the first part is just a special case of
  Corollary~\ref{cor:analytic_nondegenerate}.
  Then for each integer $k > 1$ apply Theorem~\ref{thm:bound_over_an_interval}
  with $\theta_k = 1 - \frac{1}{k}$, resulting in a sequence of subsets
  $\domfgood{k} \subset \domf$ with $\nu(\domfgood{k}) > \theta_k \nu(\domf)$
  and such that
  \[
    \nu\paren{\maxderset^n_{\bm{f}}(Q, \domfgood{k})} \ll_k
    \left( \bm{\psi}(Q) \bm{\varphi}(Q) \right)^\frac{\alpha}{n+1}
    \nu(\domfgood{k})
  \]
  for $Q$ large enough, where the implied constant is independent of $Q$.
  Therefore by condition~\eqref{eq:cnd:metric_convergence} we have
  \[
    \sum_{Q = 1}^{\infty}
      \nu\paren{\maxderset^n_{\bm{f}}(Q, \domfgood{k})}
    \ll \nu(\domfgood{k})
      \sum_{Q = 1}^{\infty}
        \paren*[\big]{ \bm{\psi}(Q)\bm{\varphi}(Q) }^\frac{\alpha}{n+1}
    < \infty
  \]
  and by the Borel-Cantelli Lemma this implies that
  $\nu\paren*[\big]{\maxderset^n_{\bm{f}}(\domfgood{k})} = 0$.

  Now observe that for every $k, Q > 1$ we have
  $\maxderset^n_{\bm{f}}(Q, \domfgood{k})
  \subseteq \maxderset^n_{\bm{f}}(Q, \domf)$,
  hence
  $\maxderset^n_{\bm{f}}(\domfgood{k})
  \subseteq \maxderset^n_{\bm{f}}(\domf)$. Thus
  \begin{align*}
  \nu\paren*[\big]{\maxderset^n_{\bm{f}}(\domf)}
    &\leq \nu\paren{\maxderset^n_{\bm{f}}(\domf)
                    \setminus \maxderset^n_{\bm{f}}(\domfgood{k})}
          + \nu\paren{\maxderset^n_{\bm{f}}(\domfgood{k})} \\
    &=    \nu\paren{\maxderset^n_{\bm{f}}(\domf)
                    \setminus \maxderset^n_{\bm{f}}(\domfgood{k})} \\
    &\leq \nu\paren{\domf \setminus \domfgood{k}} \\
    &=    \nu\paren{\domf} - \nu\paren{\domfgood{k}} \\
    &\leq \frac{1}{k} \nu\paren{\domf} \to 0
  \end{align*}
  as $k \to \infty$, and we conclude that
  $\nu\big(\maxderset^n_{\bm{f}}(\domf)\big) = 0$, as required.
\end{proof}

\section{Final remarks}%
\label{sec:final_remarks}

There is a notable gap between the hypotheses of
Theorem~\ref{thm:bernik_goetze_gusakova} and those of
Theorem~\ref{thm:lower_bound_for_manifold_count}.
For example, when $\bm{f}$ is a polynomial map our theorem only applies to at
most finitely many values of $n$. It would therefore be interesting to explore
the limit of the techniques presented in this paper, and a possible approach
would be to adapt the work of Aka, Breuillard, Rosenzweig, and de Saxc\'{e}
\cite{AkaBRS2018} to determine the precise obstruction to the applicability of
Theorem~\ref{thm:km_sublattices} to the present problem.

We also note that Theorem~\ref{thm:bound_over_an_interval} suggests that the
volume of the approximation targets plays a greater role than the length of
their sides in determining whether a certain rate of approximation is achievable
or not. In other words, we conjecture the following improvement of
Proposition~\ref{prop:manifold_khincin_point_divergence} for the set
$\dualappr_{n,\bm{f}}^*(\psi_0,\dotsc,\psi_m)$ of points $\bm{x} \in \domf$ such
that $\bm{f}(\bm{x}) \in \dualappr_{n,m+1}^*(\psi_0,\dotsc,\psi_m)$, where
the latter is the set of $\bm{x} \in \R^{m+1}$ such that
\[
  \abs{x_k - \alpha_k} < \pointapproxf[\psi_k]{\polyheight{\bm{\alpha}}}
\]
for infinitely many $\bm{\alpha} \in \Alg{m+1}{n}$.

\begin{conjecture}
  \label{conj:manifold_khincin_point_divergence}
  Let $\psi_0,\dotsc\psi_m \colon \R^+ \to \R^+$ be decreasing functions
  such that $\psi_i \in O(\psi_j)$ for every $0 \leq i < d$ and
  $d \leq j \leq m$, and suppose that there is a $\kappa > 0$ such that
  \[
    \kappa^{n-m+d} >
    \lim_{t \to \infty} \frac{\psi_d(\kappa^t) \dotsm \psi_m(\kappa^t)}
                             {\psi_d(\kappa^{t+1}) \dotsm \psi_m(\kappa^{t+1})}.
  \]
  Further, let $g$ be a dimension function such that $r^{-d}g(r)$ is
  non-increasing, and assume that $\bm{f}$ is Lipschitz continuous,
  that $\Vand(\bm{f}) \neq 0$, and
  that $\bm{f}$ satisfies condition~\eqref{eq:cnd:minors_d} on $\domf$.
  Then
  \[
    \Haus{\dualappr_{n,\bm{f}}^*(\psi_0,\dotsc,\psi_m)} =
    \begin{cases}
      0            & \text{if } S_{n,d}^g(\psi_0,\dotsc,\psi_m) < \infty \\
      \Haus{\domf} & \text{if } S_{n,d}^g(\psi_0,\dotsc,\psi_m) = \infty
    \end{cases}
  \]
  where
  \[
    S_{n,d}^g(\psi_0,\dotsc,\psi_m) \coloneqq
    \sum_{Q = 1}^{\infty}
      Q^n \frac{\psi_d(Q)\dotsm\psi_m(Q)}{Q^{m+1-d}}
      g\paren{\frac{\psi_0(Q)\dotsm\psi_{d-1}(Q)}{Q^d}}.
  \]
\end{conjecture}

Furthermore, observe that a version of \cite[Lemma~4.6]{DasFSU2018} for flows
$\domf \to \GL_{n+1}(\R)$ would allow us to extend
Theorem~\ref{thm:bound_over_an_interval} to more general measures.
In this spirit and motivated by \cite{KleinLW2004, Weiss2002, PolliV2005}, as well as
recent work by Khalil and Luethi, we propose the following:

\begin{conjecture}
  Let $\bm{f} \colon \domf \subseteq \R^d \to \R^{m+1}$ be a continuous map, and
  let $\nu$ be a measure on $\domf$ such that
  $(\Schur^{n-d} \circ \bm{f})_* \nu$ is Federer, decaying, and rationally
  non-planar. Also let $\psi \colon \R^+ \to \R^+$ be a decreasing function.
  Then for any ball $B \subseteq \domf$
  \[
    \nu\paren{\dualappr_{n,\bm{f}}^*(\psi) \cap B} = \nu(B)
    \quad \text{if} \quad
    \sum_{Q = 1}^\infty Q^{n-m-1} \psi^{m+1}(Q) = \infty.
  \]
\end{conjecture}

\printbibliography

\bigskip

\noindent
Department of Mathematics\\
University of York\\
York, YO10 5DD\\
United Kingdom\\
E-mail: \texttt{ap1466@york.ac.uk}

\end{document}